\documentclass[a4paper,11pt]{amsart}

\usepackage{geometry} 
\usepackage{amssymb}  
\usepackage{mathtools}      
\usepackage{mathabx}        
\usepackage[bb=fourier,cal=euler,scr=rsfs]{mathalfa}	
\usepackage{enumitem}       
\usepackage[colorlinks=true]{hyperref} 

\usepackage[ansinew]{inputenc}

\usepackage{comment}

\usepackage{graphicx} 

\usepackage{import}
\usepackage{xifthen}
\usepackage{pdfpages}
\usepackage{transparent}

\hypersetup{bookmarksdepth = 3} 
\numberwithin{equation}{section}     

\setlist[enumerate,1]{label={\upshape(\roman*)},ref=\roman*}
\setlist[enumerate,2]{label={\upshape(\alph*)},ref=\alph*}

\newtheorem{theorem}{Theorem}[section]
\newtheorem{teo}[theorem]{Theorem}

\newtheorem{coro}[theorem]{Corollary}

\newtheorem{lemma}[theorem]{Lemma}
\newtheorem{lema}[theorem]{Lemma}
\newtheorem{obs}[theorem]{Remark}

\newtheorem{fact}[theorem]{Fact}

\newtheorem{afir}[theorem]{Claim}
\newtheorem{prop}[theorem]{Proposition}

\newtheorem{claim}[theorem]{Claim}

\newtheorem{question}[theorem]{Question}

\newtheorem{thmintro}{Theorem}

\theoremstyle{definition}
\newtheorem{defi}[theorem]{Definition}
\newtheorem{remark}[theorem]{Remark}

\newtheorem{example}[theorem]{Example}

\newcommand{\eps}{\varepsilon}

\newcommand{\R}{\mathbb R}
\newcommand{\RR}{\R}
\newcommand{\ZZ}{\mathbb Z}
\newcommand{\Z}{\mathbb{Z}}
\newcommand{\N}{\mathbb N}
\newcommand{\HH}{\mathbb{H}}

\newcommand{\CC}{\mathbb{C}}

\newcommand{\cA}{\mathcal{A}}

\newcommand{\cF}{\mathcal{F}}

\newcommand{\cG}{\mathcal{G}}

\newcommand{\cL}{\mathcal{L}}

\newcommand{\cC}{\mathcal{C}}

\newcommand{\mt}{\widetilde M}
\newcommand{\ft}{\widetilde f}

\newcommand{\F}{\widetilde{\cF}}

\DeclareMathOperator{\diam}{diam}

\title{Invariant sets for homeomorphisms of hyperbolic 3-manifolds}

\author{Elena Gomes} 
\address{Universit\'e Paris Saclay, Paris, France}
\email{elenaagomes@gmail.com}

\author{Santiago Martinchich} 
\address{IESTA, Universidad de la Rep\'ublica, Uruguay}
\email{santiago.martinchich@fcea.edu.uy }

\author{Rafael Potrie} 
\address{Centro de Matem\'atica, Universidad de la Rep\'ublica, Uruguay  \&  IRL-IFUMI CNRS }
\email{rpotrie@cmat.edu.uy}
\urladdr{http://www.cmat.edu.uy/~rpotrie/}

\thanks{ E. G. was partially supported by a CAP scholarship and CSIC-Iniciacion. S.M. was partially supported by Fondo Vaz Ferreira (MEC) and CSIC.
 R. P. was partially supported by CSIC I+D project 'Estructuras Topol\'ogicas de sistemas parcialmente hiperb\'olicos y aplicaciones'.  }

\begin{document}

\begin{abstract}
We prove that under some assumptions on how points escape to infinity in the universal cover, homeomorphisms of hyperbolic 3-manifolds are forced to have several invariant sets (in particular, they cannot be minimal). For this, we use some shadowing techniques which, when the homeomorphism has positive speed with respect to a uniform foliation, allow us to obtain strong consequences on the structure of the invariant sets. We discuss also homological rotation sets and end the paper with some extensions to other manifolds as well as posing some general problems for the understanding of minimal homeomorphisms of 3-manifolds. 

\bigskip

\end{abstract}

\maketitle

\section{Introduction}

In the last few years we have seen many important results regarding topological dynamics of surface homeomorphisms homotopic to the identity and their relation with how orbits wind around the topology of the manifold. For this, many rotation sets have been proposed. We refer the reader to \cite{ABP,GM,GGL} and refererences therein for discussion about this subject (with emphasis in the higher genus case). 

These results have some analogies with recent developments on flows (in particular Reeb flows) in 3-manifolds and connections with some famous problems such as the existence of minimal flows in the 3-dimensional sphere. Existence and properties of minimal flows and homeomorphisms in closed 3-manifolds is a quite large and unexplored ground. We refer the reader to \cite{HT,FH,Rechtman,HR,Frankel,FishHofer} for discussions around these subjects and problems. We expand briefly on this in \S\ref{s.generalmanifolds}.

In this paper, we want to provide some results for dynamics of certain homeomorphisms of 3-manifolds, specializing in hyperbolic 3-manifolds. Before we state our results, let us pose one motivating question. Recall that a closed hyperbolic 3-manifold is a compact manifold which is the quotient of $\HH^3$ by a discrete group of isometries. These include manifolds which are obtained as suspensions of pseudo-Anosov maps of higher genus surfaces by a result due to Thurston \cite{Thurston}. It is now known that any hyperbolic 3-manifold admits a finite lift which looks this way \cite{Agol}. 

If $f: M \to M$ is a homeomorphism of a closed hyperbolic 3-manifold, then, Mostow rigidity implies that there is a finite iterate which is homotopic to the identity (see e.g. \cite[Appendix A]{BFFP}). For $f: M \to M$ homotopic to the identity, there is a selected lift $\ft: \mt \to \mt$ which we call a \emph{good lift} and is obtained by lifting the homotopy to the universal cover (in particular, it commutes with deck transformations and is bounded distance from the identity). The trigger question that motivated this work was: 

\begin{question}\label{quest-motiv}
Let $f: M \to M$ be a minimal homeomorphism of a closed hyperbolic 3-manifold which is homotopic to the identity and let $\ft: \mt \to \mt$ be a good lift. Is it true that for every $x \in \mt$ we have that $\lim_n \frac{1}{n} d_{\HH^3}(\ft^n(x),x) =0$? 
\end{question}

Recall that a minimal homeomorphism is one for which every orbit is dense. As far as we know, the only known examples of such homeomorphisms in hyperbolic 3-manifolds come from the strong foliations of Anosov flows (see \cite{HT}), but there is some indication that other examples could exist (see \cite{BFP} and \S\ref{s.generalmanifolds}). 
 
When a point $x \in \mt$ verifies that $\mathrm{liminf}_n \frac{1}{n} d_{\HH^3}(\ft^n(x),x) >0$ we say that it has \emph{positive escape rate}. We discuss this notion and this question in more depth in \S \ref{ss.escaperateI} where we relate it with the notion of homological rotation set, introduced in \cite{Shwartz} and which has been one of the guiding concepts in the study of surface dynamics homotopic to the identity. We refer the reader to \cite{ABP,GM,GGL} for recent works relating homotopical and homological rotation sets in hyperbolic surfaces and which describe the history of the subject in more depth. 

Motivated by this, we set us as a goal to obtain results that ensure some compact invariant subsets of homeomorphisms of 3-manifolds and to see if under some assumptions we can understand some structure of these invariant sets. 

We present here some of the main results of this work and later we will give other more general results and consequences. For the sake of clarity, we will restrict to closed hyperbolic 3-manifolds in this introduction. 

The first result shows that minimality is incompatible with a condition which forces points to have uniformly positive escape rate. This result should be compared with the work of Frankel \cite{Frankel} which studies this property for flows. To state the result, let us give a definition: Let $f: M \to M$ be a homeomorphism of a closed hyperbolic 3-manifold which is homotopic to the identity, we say it is \emph{quasi-geodesic} if there exists a constant $\lambda>1$ such that for every $x \in \mt$ and $n>0$ we have that: 

\begin{equation}\label{eq:QG}
 \lambda^{-1} n - \lambda < d_{\HH^3}(\ft^n(x), x) < \lambda n + \lambda. 
\end{equation}

Note that this property implies that for $x \in \mt$ the map $\ZZ \to \mt \cong \HH^3$ given by $n \mapsto \ft^n(x)$ is a quasi-isometry and we will show later that these two properties are equivalent (i.e. that the quasi-isometry constants can be shown not to depend on $x$).  

\begin{thmintro}\label{teoA}
Let $f: M \to M$ be a quasi-geodesic homeomorphism of a closed hyperbolic 3-manifold. Then, $f$ contains infinitely many disjoint compact invariant sets and has positive topological entropy. 
\end{thmintro}

Using this result and a recent result from \cite{BPS} one can obtain a partial answer to Question \ref{quest-motiv}:

\begin{thmintro}\label{teoB}
Let $f: M \to M$ be a minimal homeomorphism of a closed hyperbolic 3-manifold homotopic to the identity. Then, there is a $G_\delta$-dense\footnote{I.e. containing a countable intersection of open and dense sets.} set of points $\cG \subset M$ so that if $x \in \cG$ and $\tilde x \in \mt$ is a lift of $x$ then: 
$$ \liminf_{n \to + \infty} \frac{1}{n} d_{\HH^3}(\ft^n(\tilde x),\tilde x) = 0. $$
\end{thmintro}

For flows, in \cite{Frankel} more information than that given by Theorem \ref{teoA} is obtained about these compact sets (one can show these are periodic orbits of the flow), but such a result cannot hold for homeomorphisms (see Example \ref{ex:pseudocircle}).  One can still expect to say more about the structure of these sets, but we achieve this only for some particular class of quasi-geodesic homeomorphisms. The following result is the technical core of this paper and extends ideas going back to \cite{Handel} which were developed in some particular cases in \cite{BFFP}. See also the recent \cite{FrLa} for further important developments.

\begin{thmintro}\label{teoC}
Let $f: M \to M$ be a homeomorphism homotopic to the identity of a closed hyperbolic 3-manifold which has positive escape rate with respect to a uniform $\RR$-covered foliation $\cF$. Then, $f$ has uncountably many disjoint closed invariant sets, each of which satisfying that when lifted to the universal cover, it contains a connected component intersecting every leaf of $\widetilde{\cF}$ in a compact set. 
\end{thmintro}

Here we need to precise some terms. Given a foliation $\cF$ of a closed 3-manifold, we say it is \emph{uniform}-$\RR$-\emph{covered} if when lifted to the universal cover, we have that $\widetilde{\cF}$ verifies that every leaf is a properly embedded plane in $\mt$ and for every pair of leaves $L,L'$ we have that the Hausdorff distance between $L$ and $L'$ is bounded (see \cite{FPMin} for discussion, and below we give some more equivalences of this notion). Such foliations are quite abundant in hyperbolic 3-manifolds and extend the notion of fibrations for fibered hyperbolic 3-manifolds: they are always blow ups of slitherings (see \cite{ThurstonCircle1,Calegari-book}). The reason for calling them $\RR$-covered is because in particular, we can see that the leaf space $\cL_\cF = \mt/_{\cF}$ is homeomorphic to $\RR$. 

We say that $f$ has \emph{positive escape rate} with respect to $\cF$ if for every $x \in \mt$ and a good lift $\ft$ of $f$ we have that the leaf of $\ft^n(x)$ goes to $+\infty$ in $\cL_{\cF}$. 

As an application of Theorem \ref{teoC}, in \S~\ref{s.homological} we show that if a homeomorphism has positive speed with respect to some homological direction then it must have many invariant sets. We note here that the invariant sets that we produce have all at least topological dimension one\footnote{This is optimal, one cannot ensure that the invariant sets are smaller, or that there are periodic orbits. For instance, one can consider the suspension flow of a pseudo-Anosov homeomorphism of a surface and take $f$ to be an irrational time of this suspension. The smallest closed invariant sets are circles, and the homeomorphism is quasi-geodesic (see Proposition \ref{prop.qg}).} (see Proposition \ref{prop-Tgammacon}). In Example \ref{ex:pseudocircle} we show that these invariant sets can still be quite wild. 

In the next section we give more precise definitions and some preliminary results that will allow us to state our main results in more generality. In \S \ref{ss.escaperateI} we present some notions of escape rate and indicate the strategy to attack Theorem \ref{teoB} which is provided in \S~\ref{ss.teobandex}. In \S~\ref{ss.homological}  we state some results relating with homological rotation sets that are studied later in \S \ref{s.homological} as a consequence of Theorem C. In \S \ref{ss.presQG} we introduce quasi-geodesic homeomorphisms and some equivalences and state a result which implies Theorem \ref{teoA}. The proof of Theorem \ref{teoA} is done in \S \ref{s.QG}. In \S \ref{ss.RcovI} we state some precise results that imply Theorem \ref{teoC} as well as some of the intermediate results to indicate the strategy which is carried out in \S \ref{s.positivefol}. Finally, in \S\ref{s.generalmanifolds} we state some extensions to other 3-manifolds and propose some problems related to homeomorphisms and flows with positive escape rate.  

{\small \emph{Acknowledgements:} Part of this paper is the content of the master thesis \cite{Gomes} of the first author, made in PEDECIBA-Udelar in Uruguay. The authors would like to thank Ian Agol, Alfonso Artigue, Jairo Bochi, Sylvain Crovisier, Sergio Fenley, Pablo Lessa, Ana Rechtman, Jonathan Zung for useful comments and exchange. We thank the referee for some insightful comments, including some questions that we have included in the text.}

\section{Precise statement of results and preliminaries} 
Here we present the main results of the paper and the notions involved. We will restrict to the case of closed hyperbolic 3-manifolds. Statements in more generality can be found in \S~\ref{s.generalmanifolds}. 

Let $M$ be a hyperbolic 3-manifold. We always consider the hyperbolic metric on $\widetilde M \cong \HH^3$. For two points $x, y \in \widetilde M$, we write $d(x, y)$ to denote the distance between them given by this metric. We will also write $d(A, B)$ to denote the infimum of distances between points of two subsets $A$ and $B$ of $\widetilde M$.

Here $f: M \to M$ will always denote a homeomorphism of $M$ homotopic to the identity and $\ft: \mt \to \mt$ will always denote a \emph{good lift} (i.e. which commutes with deck transformations, cf. \cite[Def. 2.3, Rem.2.4]{BFFP}). 

\subsection{Escape rate}\label{ss.escaperateI}

Consider the function $x \mapsto d(\ft(x),x)$ defined on $\mt$. Since $\ft$ commutes with deck transformations, which act as isometries, this function is $\pi_1(M)$ equivariant (and thus defines a function $\varphi: M \to \RR_{\geq 0}$). Similarly, we can define $x \mapsto d(\ft^n(x),x)$ which induces a function $\varphi^{(n)}: M \to \RR_{\geq 0}$. The sequence $\varphi^{(n)}$ is \emph{subaditive}, that is $\varphi^{(n+m)} \leq \varphi^{(n)} \circ f^m + \varphi^{(m)}$ by the triangle inequality. This implies that for every ergodic $f$-invariant measure $\mu$ we have a well defined \emph{escape rate} defined as $\ell_\mu = \inf_n \frac{1}{n} \int \varphi^{(n)} d\mu$, which coincides with the limit $\ell_\mu= \lim_n \frac{1}{n} \varphi^{(n)}(x)$ for $\mu$-almost every $x \in M$ (see \cite[\S 1.2]{GGL} for a detailed presentation, we will use the facts presented there below).  

We say that $f$ \emph{escapes to infinity with uniform positive rate} $\tau>0$ if for every $f$-invariant ergodic measure $\mu$ we have that $\ell_\mu\geq \tau$. 

Note that if $f$ is quasi-geodesic with constant $\lambda>0$ (cf. \eqref{eq:QG}), then we have that $\ell_\mu \geq \lambda^{-1}$ for every ergodic $\mu$, so it escapes to infinity with uniform positive rate $\tau=\lambda^{-1}$. Applying a general result on linear cocycles due to \cite{BPS} we can show that if every point $x \in \mt$ verifies that $\liminf_{n} \frac{1}{n} d(\ft^n(x),x))>0$ and $f$ is minimal, then $f$ has to be quasi-geodesic, which due to Theorem \ref{teoA} gives a contradiction. We will expand this in \S~\ref{ss.teobandex} to prove Theorem \ref{teoB} and give an example showing that if the homeomorphism is not minimal, it is possible to have that every point $x \in \mt$ verifies $\liminf_{n} \frac{1}{n} d(\ft^n(x),x))>\tau>0$ without $f$ being quasi-geodesic. 

In some cases, one can promote the property of escaping to infinity with uniform positive rate. Assume there is a (continuous) function $Q: \mt \to \RR$ which satisfies that:

\begin{enumerate}
\item There exists $k>0$ such that for every $\gamma \in \pi_1(M)$ and points $x,y \in \mt$ one has that: 
\begin{equation}\label{eq:QM}
|Q(\gamma x) - Q(x) + Q(y) - Q(\gamma y)| < k.
\end{equation}
\end{enumerate}

This should be compared with \emph{quasi-morphisms} (see \cite{Calegari-scl} for more information). 


We say that $f$ has \emph{positive escape rate with respect} to $Q$ if one has that $Q(\ft^n(x)) \to + \infty$ for every $x \in \mt$. 

We will show in \S~\ref{ss.rateQM} that the following holds: 

\begin{prop}\label{prop-QM}
If $f$ has positive escape rate with respect to $Q$ it holds that  given $k>0$ there is some $n_0$ so that if $n>n_0$ we have that $Q(\ft^n(\tilde x)) - Q(\tilde x) > k$ for every $\tilde x \in \mt$. In particular, there is $\lambda>0$ so that $\mathrm{liminf}_n \frac{1}{n} Q(\ft^n(\tilde x)) > \lambda >0$. Moreover, $f$ is a quasi-geodesic homeomorphism. 
\end{prop}


\subsection{Homological rotation}\label{ss.homological}

It is sometimes convenient to work with homological rotation sets which have better properties (for instance, positive escape rate for every invariant measure, implies uniform speed independent of the point). In particular, they satisfy the property in \eqref{eq:QM}.

Let $f: M \to M$ be a homeomorphism homotopic to the identity in a closed hyperbolic 3-manifold and let $f_t: M \to M$ be a homotopy so that $f_0=\mathrm{id}$ and $f_1=f$. We define $f_t$ for $t \in [n,n+1)$ as $f_{t-n} \circ f^n$ so that $f_t$ is defined for $t \in [0,\infty)$. For a given cohomology class $c \in H^1(M, \RR)$ we can define the \emph{escape rate with respect to} $c$ for an ergodic measure $\mu$ (or the \emph{homological rotation set in the direction of $c$}) as follows: let $\alpha \in c$ be a closed $1$-form (we are identifying the usual cohomology with the de Rham cohomology) and for a given $x \in M$ we define $\eta_x^n: [0,n] \to M$ as $\eta_x^n(t)= f_t(x)$. So, we can consider the sequence of functions $R^{(n)}_c: M \to \RR$ as:

$$ R^{(n)}_c(x) = \int_{\eta_x^n} \alpha = \sum_{i=0}^{n-1} R_c^{(1)}(f^i(x)). $$ 

This allows one to define, using Birkhoff ergodic theorem, for a given $f$-invariant ergodic measure $\mu$ the \emph{escape rate of $\mu$ with respect to $c$} as: 

$$ \ell_c(\mu) := \int R_c^{(1)}d \mu = \lim_n \frac{1}{n} R^{(n)}_c(x) \quad \mu-\text{a.e. } x. $$

\begin{remark}\label{rmk-Qc}
Note that one can define a function $Q_c: \mt \to \RR$ by considering a marked point $\tilde x \in \mt$ and lifting the $1$-form $\alpha$ to a (closed) $1$-form $\tilde \alpha$ in $\mt$ and define $Q_c(x)= \int_{\eta_x} \alpha$ where $\eta_x: [0,1] \to \mt$ is a curve such that $\eta_x(0)= \tilde x$ and $\eta_x(1)=x$. This function is well defined because $\tilde \alpha$ is closed. Note that this function satisfies equation \eqref{eq:QM} for every $K$ and we have that $Q_c(\ft^n(x))- Q_c(x) = R^{(n)}_c(x)$. 
\end{remark}

In that sense, we can show the following: 

\begin{teo}\label{teo-homology}
Let $f: M \to M$ be a homeomorphism homotopic to the identity in a closed hyperbolic 3-manifold. Assume that there is some cohomology class $c \in H^1(M,\RR)$ for which it holds that every ergodic invariant measure $\mu$ satisfies that $\ell_c(\mu)>0$. Then $f$ has uncountably many disjoint closed invariant sets and positive topological entropy. 
\end{teo}

The proof of this theorem is an application of Theorem \ref{teoC} together with some properties of surfaces in hyperbolic 3-manifolds. As noted by the referee, it could be that the previous result holds only assuming that for every ergodic $\mu$ one has that $\ell_c(\mu)\neq 0$ (it certainly does if one assumes this for every invariant measure thanks to ergodic decomposition), but we have not been able to show this nor to produce a counterexample. One should also point out that in the case of flows, thanks to \cite{Shwartz} the assumptions imply that the flow is a suspension of some map in a surface, and since $M$ is hyperbolic the map is homotopic to pseudo-Anosov and this is enough to conclude. Here, the challenge is to work with homeomorphisms. 

As a consequence we get the following result in the direction of Question \ref{quest-motiv}: 

\begin{coro}\label{coro-uniquely}
Let $f: M \to M$ be a uniquely ergodic homeomorphism of a closed hyperbolic manifold homotopic to the identity, then, for every $\pi: \hat M \to M$ finite cover, we have that the lift $\hat f: \hat M \to \hat M$ satisfies that for every $c \in H^1(\hat M, \RR)$, if $\hat \mu$ is the lift of the unique invariant measure of $\mu$, then $\ell_c(\hat \mu)=0$. 
\end{coro}

In particular, thanks to the results of \cite{AgolVF,Agol} it makes sense to ask the following: 

\begin{question}\label{quest-finite}
Let $f: M \to M$ be a homeomorphism which escapes to infinity with uniform positive rate\footnote{I.e. there exists $\tau>0$ such that $\ell_\mu\geq \tau$ for every $f$-invariant ergodic measure $\mu$.}. Is it true that there exists a finite cover $\hat M$ of $M$ such that for the lift $\hat f$ of $f$ to $\hat M$ there is a cohomology class $c \in H^1(\hat M, \RR)$ for which $\hat f$ has uniform positive escape rate\footnote{I.e. for every $\hat f$-invariant ergodic measure $\hat\mu$ one has that $\ell_c(\hat\mu)>0$.} with respect to $c$? 
\end{question}

We will discuss more on this question and related ones in \S\ref{s.homological} where we also prove Theorem \ref{teo-homology} and Corollary \ref{coro-uniquely}. Questions \ref{quest-motiv} and \ref{quest-finite} naturally raise the question of how orbits of a minimal homeomorphism of a hyperbolic manifold behave. Note that one can easily see: 

\begin{remark}
If $f: M \to M$ is a minimal homeomorphism of a closed $3$-manifold homotopic to the identity and $\ft: \mt \to \mt$ is a good lift, then, for every $x \in \mt$ we have that  the sequence $d(\ft^n(x),x)$ is unbounded. In fact, if for some point $x$ this sequence is bounded, we can consider the closure of its $\ft$-orbit, the boundary of which is a compact $\ft$-invariant set with empty interior. Thus, its projection to $M$ is a closed, proper, $f$-invariant set, contradicting minimality. In the known examples in hyperbolic 3-manifolds, minimal homeomorphisms verify that every point has escape rate of the order of $d(\ft^n(x),x) \sim \log n$. We do not know if it is possible to expect that \emph{every} minimal homeomorphism of a hyperbolic 3-manifold has escape rate slower than $n^\alpha$ for some (every) $\alpha >0$. 
\end{remark}

\subsection{The quasi-geodesic case}\label{ss.presQG} 

One way to ensure that orbits escape with uniform positive rate is to ask for the homeomorphism to be \emph{quasi-geodesic}: we say that $f: M \to M$ is quasi-geodesic if for every $x \in \mt$ we have that the map from $\ZZ$ to $\mt$ given by:
$$ n \mapsto \ft^n(x) ,$$ 
is a $\ZZ$-\emph{quasi-geodesic}, that is, there is $\lambda_x>0$ such that for every $n,m \in \ZZ$ we have that $\lambda_x^{-1} |n-m| - \lambda_x < d(\ft^n(x), \ft^m(x)) < \lambda_x |n-m| + \lambda_x$ (note that the important inequality is the first, as the second one is always verified  with a constant $\lambda_x$ independent of $x$ when $\ft$ is a good lift\footnote{Note also that if $f$ is not homotopic to the identity then no lift can verify the second inequality with a constant independent of $x$.}). 

The proof of \cite[Lemma 10.20]{Calegari-book} adapts directly to deduce: 

\begin{prop}\label{prop-uniformQG}
If $f$ is quasi-geodesic, then, one can choose the constant $\lambda_x$ to be independent on $x$. 
\end{prop}

In particular, one deduces that $f$ escapes to infinity with uniform positive rate. For this class of homeomorphisms we are able to find many compact invariant sets by using properties of shadowing of quasi-geodesics. 

\begin{teo}\label{teo-QGhomeo}
For every quasi-geodesic homeomorphism on a compact hyperbolic manifold we can associate a closed set $\Lambda_f$ of $T^1M$ invariant under the geodesic flow which contains infinitely many disjoint compact invariant sets, corresponding to distinct compact $f$-invariant sets. In particular, $f$ cannot be minimal. \end{teo}

This proves part of Theorem \ref{teoA}. We will also see in \S~\ref{ss.entropy} that the set $\Lambda_f$ has positive topological entropy with respect to the geodesic flow and the same holds for $f$. We will defer the definition of topological entropy to \S~\ref{ss.entropy}. 

The invariant set is constructed by using the classical Morse Lemma \cite[Lemma 1.24]{Calegari-book}: 

\begin{prop}\label{morselema}
For every $\lambda>0$ there is $R>0$ such that every $\ZZ$-quasi-geodesic of constant $\lambda$ in $\HH^3$ is contained in the $R$ neighborhood of a unique complete (oriented) geodesic in $\HH^3$ (we say that this geodesic \emph{shadows} the orbit).  Moreover, this geodesic is the unique which remains at bounded distance from the quasi-geodesic. 
\end{prop}

The invariant set for the geodesic flow announced in Theorem \ref{teo-QGhomeo} will be the projection of the union of all the geodesics shadowing some orbit of $\ft$. 

This will be proved in \S~\ref{s.QG}. Let us mention that Theorem  \ref{teo-QGhomeo} (as well as Theorem \ref{teoA}) are true in higher dimensions with the same proof. 

\subsection{$\RR$-covered foliations}\label{ss.RcovI}
Let $\cF$ be a foliation of $M$. We will be working with foliations of class $C^{0,1}$ (i.e. continuous foliations with $C^1$-leaves tangent to a continuous distribution), thanks to \cite{Calegari} this is no loss of generality. 

Suppose $\cF$ is a foliation of $M$ and $\F$ is its lift to $\widetilde M$.  Since leaves are $C^1$, we can measure distances within the leaves of $\F$ using the metric induced on them by $\widetilde M$. If $x$ and $y$ are two points in the same leaf $L \in \F$, we denote by $d_L(x, y)$ the distance between $x$ and $y$ \textit{within the leaf $L$}. We also denote by $d_L(A, B)$ the infimum of distances within $L$ between points of the subsets $A$ and $B$ of $L$.

We will also often use the following notation, for $\varepsilon > 0$, $L \in  \F$, $X \subset L$, and $Y \subset \widetilde M$.

\begin{equation*}
    \begin{aligned}
        B(Y, \varepsilon) &:= \{ x \in \widetilde M : d(x, Y) < \varepsilon\} \\
        B_L(X, \varepsilon) &:= \{ x \in L : d_L(x, X) < \varepsilon \}.
    \end{aligned}
\end{equation*}

If the foliation has no torus leaves (more generally, if it is Reebless) we know that the leaf space $\cL = \mt/_{\F}$ is a simply connected (possibly non-Hausdorff) one-dimensional manifold. We say that $\cF$ is $\RR$-\emph{covered} if $\cL$ is Hausdorff (equivalently, homeomorphic to $\RR$). We say that $\cF$ is \emph{uniform} if for every pair of leaves $L,L' \in \F$ there is $C>0$ such that $L$ is contained in the $C$-neighborhood of $L'$ and viceversa. See \cite[Chapter 9]{Calegari-book} for more on these foliations, in particular we will use \cite[Thm. 9.15 and Lemma 9.10]{Calegari-book}.

\begin{prop}
\label{p.existeZ}
If $\F$ is $\R$-covered and uniform, there exists a homeomorphism $Z: \cL \to \cL$ that commutes with the action of $\pi_1(M)$ on $\cL=\widetilde{M} / \F$, and a constant $c > 0$ such that for every leaf $L \in \F$, we have the bound $d(L, Z(L)) > c$.
\end{prop}

Such a $Z$ is called a \textit{structure map} of $\F$. 

Suppose now $\cF$ is $\R$-covered and uniform, and $f: M \to M$ is a homeomorphism homotopic to the identity and  $\widetilde{f}$ its good lift. Fix an identification of the leaf space with $\R$, so it makes sense to say that one leaf $L$ of $\F$ is \textit{above} another $L'$, and denote it by $L > L'$. For each $x \in \widetilde{M}$, consider the sequence of leaves $L_n(x)$ where $L_n(x)$ is the leaf through $\widetilde{f}^n(x)$.

\begin{defi}
\label{defi positive rate w.r.t. foliation}
We say that $f$ has \textit{positive escape rate with respect to $\cF$} if for every $x \in \widetilde{M}$, $\limsup_n L_n(x) = +\infty$ in the leaf space.
\end{defi}

The fact that we call this positive escape rate is because it is not hard to show the following: 

\begin{prop}\label{prop.qg}
If $f$ has positive escape rate with respect to a uniform $\RR$-coverered foliation $\cF$ then $f$ is quasi-geodesic. 
\end{prop}

With this we can now state the result in the direction of Theorem \ref{teoC}. 

\begin{teo}\label{teo.B1}
Let \( \cF \) be a uniform \( \R \)-covered foliation of a compact hyperbolic 3-manifold \( M \). If \( f: M \to M \) is a homeomorphism with positive escape rate with respect to \( \cF \), then there exist uncountably many non-empty, disjoint, \( f \)-invariant compact sets in \( M \). 
\end{teo}

The key tool to prove this result will be the existence of regulating pseudo-Anosov flows found in \cite[Thm. 9.31]{Calegari-book}, \cite{FenleyRcov}, based on \cite{ThurstonCircle1}. This will be expanded and precised in \S \ref{ss.PA}. The main difficulty is that we cannot work leafwise since leaves are not preserved, so we need to be more careful in the way we construct the invariant sets. 

\begin{teo}[Existence of a pseudo-Anosov regulating flow]
\label{thm-regulating}
    If \( M \) is a hyperbolic 3-manifold and \( \cF \) is an \( \R \)-covered foliation of \( M \), then there exists a pseudo-Anosov flow \( \phi_t: M \to M \), transverse to \( \cF \), such that every orbit of its lift \( \widetilde{\phi}_t:\widetilde{M}\to \widetilde{M} \) intersects every leaf of \( \F \). 
\end{teo}

To prove Theorem \ref{teo.B1}, and to obtain the desired properties about the closed invariant sets, we will need to have more information on the topology of the invariant sets. We will base the proof on the ideas of \cite{BFFP} (which shows the result for a homeomorphism \emph{preserving} the foliation $\cF$), which are in turn modeled in the classical result of Handel \cite{Handel} (see also \cite{Fathi,BFFPT1S,Militon}). The main difficulty here is that we cannot work leafwise, and since $f$ is homotopic to the identity we need to control simulataneously the progress and the transverse geometry to obtain some coarse hyperbolicity. 

The proof of Theorem \ref{teoC} then splits into some steps, the first of which is the most challenging: 

\begin{prop}
\label{prop exists T sub gamma}
    For every $\gamma \in \pi_1(M)$ associated with a regular periodic orbit of $\phi_t$, there exists a closed set $T_\gamma \subset \widetilde M$, invariant under $\gamma$ and $\widetilde f$, intersecting every leaf of $\F$ in a non-empty compact set. Moreover, there exists $r_0$ (independent of $\gamma$) such that, for every $r \geq r_0$, the maximal invariant closed set within the $r$-neighborhood of $g_\gamma$ (the geodesic associated to $\gamma$) is $T_\gamma$. Moreover, $T_\gamma$ contains a closed connected set intersecting every leaf of $\F$. 
\end{prop}

\begin{figure}[ht]
    \centering
    \def\svgwidth{0.7\columnwidth}
\begingroup%
  \makeatletter%
  \providecommand\color[2][]{%
    \errmessage{(Inkscape) Color is used for the text in Inkscape, but the package 'color.sty' is not loaded}%
    \renewcommand\color[2][]{}%
  }%
  \providecommand\transparent[1]{%
    \errmessage{(Inkscape) Transparency is used (non-zero) for the text in Inkscape, but the package 'transparent.sty' is not loaded}%
    \renewcommand\transparent[1]{}%
  }%
  \providecommand\rotatebox[2]{#2}%
  \newcommand*\fsize{\dimexpr\f@size pt\relax}%
  \newcommand*\lineheight[1]{\fontsize{\fsize}{#1\fsize}\selectfont}%
  \ifx\svgwidth\undefined%
    \setlength{\unitlength}{352.64131037bp}%
    \ifx\svgscale\undefined%
      \relax%
    \else%
      \setlength{\unitlength}{\unitlength * \real{\svgscale}}%
    \fi%
  \else%
    \setlength{\unitlength}{\svgwidth}%
  \fi%
  \global\let\svgwidth\undefined%
  \global\let\svgscale\undefined%
  \makeatother%
  \begin{picture}(1,0.52968407)%
    \lineheight{1}%
    \setlength\tabcolsep{0pt}%
    \put(0,0){\includegraphics[width=\unitlength,page=1]{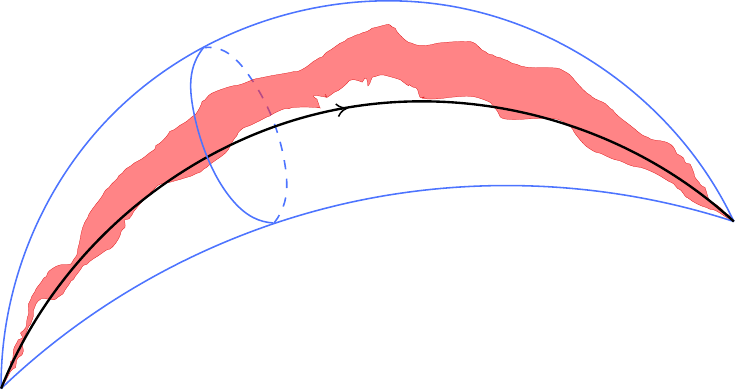}}%
    \put(0.35458648,0.45501222){\color[rgb]{0.85490196,0,0.01176471}\makebox(0,0)[lt]{\lineheight{10.25}\smash{\begin{tabular}[t]{l}$T_\gamma$\end{tabular}}}}%
  \end{picture}%
\endgroup%

    \caption{Ilustration of Proposition \ref{prop exists T sub gamma}.}
    \label{Tgamma}        
\end{figure}

The last statement of the theorem is proved in \S~\ref{ss.furtherprops} (see Proposition \ref{prop-Tgammacon}) where we will also discuss more on the possible topologies that $T_\gamma$ can have (in particular, that it may not contain any arc but has separating properties similar to those of circles).

An element $\alpha \in \pi_1(M)$ is called \textit{primitive} when the equality $\alpha = \beta^j$ for some $\beta \in \pi_1(M)$ and $j \geq 0$ implies that $j = 1$.

\begin{prop}
\label{prop if eta and gamma different, T projections are disjoint}
    If $\gamma$ and $\eta$ are primitive elements of $\pi_1(M)$ associated with distinct regular periodic orbits of $\phi_t$ then $\pi(T_\gamma)$ is disjoint from $\pi(T_\eta)$.
\end{prop}

Finally, the fact that there are uncountably many invariant sets follows by extending the shadowing property from periodic orbits to general orbits of the pseudo-Anosov flow. We achieve this in \S~\ref{ss-exclosure} by using the ideas on the previous section. 

To prove the results presented here, another important tool will be the following theorem which allows us to make geometric arguments along the leaves of the foliation (see \cite{Calegari-book}). We record its statement here for future use. 

\begin{teo}[Candel's uniformization theorem]
\label{teo 'Candel'}
    Let $\cF$ be an $\R$-covered foliation of a compact hyperbolic 3-manifold. There exists a constant $C > 1$ such that every leaf of $\F$ is $C$-quasi-isometric to the hyperbolic plane $\mathbb H^2$.
\end{teo}

Recall that a map $h: X \to Y$ between metric spaces is a $C$-quasi-isometry if its image is $C$-dense and for every $x,x' \in X$ we have that $C^{-1} d_X(x,x') - C < d_Y(h(x),h(x')) < C d_X(x,x') + C$. 


\section{Quasi-geodesic homeomorphisms}\label{s.QG}

In this section we consider a homeomorphism $f: M \to M$ of a hyperbolic 3-manifold which is homotopic to the identity and such that it is \emph{quasi-geodesic} in the sense of \S\ref{ss.presQG}. The goal of this section is to prove Theorem \ref{teo-QGhomeo} and that such homeomorphisms must have positive topological entropy. Together, this will imply Theorem \ref{teoA}. We note here that while all this section is stated in the 3-dimensional case, it works equally well for homeomorphisms of hyperbolic $n$-manifolds (or even manifolds admitting negatively curved metrics), we leave the verification of this to the interested reader. 



\subsection{Constructing the invariant set} 

Recall that we can identify the set of (oriented) geodesics of $T^1\HH^3 \cong T^1\mt$ with the pairs of distinct (ordered) points of $\partial \HH^3 \cong S^2$. 

We can define maps $e^+: \mt \to \partial \HH^3$ and $e^-: \mt \to \partial \HH^3$ as in \cite[\S 10]{Calegari-book} so that for every $x \in \mt$ the points $e^+(x)$ and $e^-(x)$ are the positive and negative endpoints of the geodesic given by Proposition \ref{morselema}. 

Let us prove some elementary properties of these maps: 

\begin{lema}\label{lem-emascont}
The maps $e^+$ and $e^-$ are continuous. 
\end{lema}

\begin{proof}

Let $x_k \to x_\infty$ in $\mt$ and let $\ell_k \subset \HH^3$ be the complete geodesic of $\HH^3$ which satisfies that $\{\ft^n(x_k)\}_n$ is contained in the $R$-neighborhood of $\ell_k$. 

Since $\ell_k$ has a point at distance $\leq R$ from $x_k$ it follows that for every $k$ large we have that $\ell_k$ intersects $B(x_\infty, R+1)$. Thus, up to taking a subsequence, we can assume that $\ell_k \to \ell_\infty$  uniformly on compact sets where $\ell_\infty$ is a complete geodesic. 

Since $\ft$ is continuous, for every $m>0$ we have that there is $k_m$ such that for every $|j|<m$ we have that $d(\ft^j(x_k), \ft^j(x_\infty))<1$ for every $k>k_m$. 

Fix some $m$ and let $J_k$ be a compact arc of $\ell_k$ which contains the points $A_{m,k}=\{\ft^i(x_k)\}_{i=-m}^m$ in its $R$-neighborhood. Note that we can choose $J_k$ to have uniformly bounded length, since the set $A_{m,k}$ is contained in a neighborhood of size $Cm + C$ of $x_k$ (where $C$ is the uniform quasi-geodesic constant of the orbits of $x_k$). Considering $k$ large enough, we get that the points of $A_{m,k}$ are at distance $1$ of the corresponding points $\{\ft^i(x_\infty)\}_{i=-m}^m$ and that the arc $J_k$ is in the neighborhood of radius $1$ of $\ell_\infty$. We deduce that the points $\{\ft^i(x_\infty)\}_{i=-m}^m$ are contained in the $R+2$ neighborhood of $\ell_\infty$. Since this was independent of $m$ we deduce that the full $\ft$ orbit of $x_\infty$ is at distance $R+2$ of $\ell_\infty$ which implies that $\ell_\infty$ is actually the shadowing geodesic for $x_\infty$. 

Since the limit points of geodesics vary continuously with the geodesic (with the uniform convergence in compact sets) we deduce that $e^+$ and $e^-$ are continuous. Indeed, we have shown that every converging subsequence of $e^+(x_k)$ and $e^-(x_k)$ converges to $e^+(x_\infty)$ and $e^-(x_\infty)$, respectively, which implies that $e^+(x_k) \to e^+(x_\infty)$ and $e^-(x_k) \to e^-(x_\infty)$.

\end{proof}

\begin{lema}\label{lem-finvariant}
The maps $e^+$ and $e^-$ are $\ft$-invariant, that is, $e^\pm \circ \ft = e^\pm$. 
\end{lema}
\begin{proof}
This is direct, since the $\ft$ orbit of $x$ coincides with the $\ft$ orbit of $\ft(x)$. 
\end{proof}

\begin{lema}\label{lem-pi1equiv}
The maps $e^+$ and $e^-$ are $\pi_1(M)$-equivariant, that is, if $\gamma \in \pi_1(M)$ we have that $e^\pm(\gamma x) = \gamma e^{\pm}(x)$ where in the right hand side we are considering the induced action of $\pi_1(M)$ on $\partial \HH^3$ (recall that $\pi_1(M)$ acts by isometries of $\HH^3$).   
\end{lema}
\begin{proof}
This follows from uniqueness of the shadowing geodesic and the fact that $\ft$ commutes with deck transformations. 
\end{proof}

Since the maps are continuous and equivariant, and $\mt$ is connected, we deduce: 

\begin{coro}\label{cor-nonconst}
The maps $e^+$ and $e^-$ are non constant and their images are dense and connected in $\partial \HH^3$. 
\end{coro} 

\begin{proof}
The fact that the image is connected follows from continuity of $e^\pm$. Density of the image (which implies that it is non-constant) follows from the fact that $\pi_1(M)$ acts minimally on $\partial \HH^3$ (see \cite{Thurston}). 
\end{proof}

Note that we do not claim that the image of these maps is closed, so they need not be surjective. These properties are analogous to some properties verified for quasi-geodesic flows shown in \cite[\S 10.7]{Calegari-book}. A more detailed study of the maps $e^{\pm}$ when one studies a \emph{quasi-geodesic flow} rather than a homeomorphism can be found in \cite{Frankel,FrLa}.

Now we can define the set $\Lambda_f \subset T^1M$ associated to $f: M \to M$.

Recall that $T^1M$ denotes the set of unit vectors for the hyperbolic metric tangent to some point in $M$. That is, we have $T^1M = \{ (p,v) \ : \ p \in M \ , \ v \in T_pM \ , \|v\|=1\}$. Similarly, we define $T^1\mt \cong T^1\HH^3$. The fundamental group of $M$ acts naturally on $T^1\mt$ (because its elements act as isometries of $\HH^3$) and one can easily see that $T^1M = T^1\mt/_{\pi_1(M)}$. Given $(p,v) \in T^1M$ (or in $T^1\mt$) we have a unique geodesic $g_{p,v}: \RR \to M$ (or  $g_{p,v}: \RR \to \mt$) which is parametrized by arc length and such that $g_{p,v}(0)=p$ and $g_{p,v}'(0)=v$. The geodesic flow $G_t: T^1M \to T^1M$ is given by $G_t(p,v) = (g_{p,v}(t), g_{p,v}'(t))$ and is standard to check that this is a flow. We denote by $\widetilde G_t$ its lift to $T^1\mt$, which is also the geodesic flow of $\mt$. 

Given a point $(p,v) \in T^1\mt$ we can consider the points $v_+$ and $v_-$ in $\partial \HH^3$ to be the forward and backward limit points in $\mt \cong\HH^3$ of $g_{p,v}(t)$ as $t \to \pm \infty$. 

Then, we can define the following subset of $T^1\mt$ associated to $f$. 

\begin{equation}\label{eq:lambdaf}
\widetilde{\Lambda_f} = \{ (p, v) \in T^1\mt \ : \ \exists x \in \mt \text{ such that } e^-(x)=v_- \ , \ e^+(x)=v_+ \} 
\end{equation}

While the association  $f \mapsto \Lambda_f$ may not be locally constant (even in the case of quasi-geodesic flows), it could be that, as in the case of flows, there is a \emph{geometric core} which is independent of the map as long as one varies $f$ continuously (see \cite{FrLa} for precise statements in the case of flows). It could be interesting to study this further for homeomorphisms.  For our purposes, we will only need the following:

\begin{lema}\label{lem.defLambda}
The set $\widetilde{\Lambda_f}$ is closed, $\widetilde{G_t}$-invariant and $\pi_1(M)$-invariant. Therefore, its projection to $T^1M$ defines a $G_t$-invariant compact set that we denote $\Lambda_f$. 
\end{lema}

\begin{proof}
This is a direct consequence of the properties we proved for the functions $e^+$ and $e^-$ in Lemmas \ref{lem-emascont} and \ref{lem-pi1equiv}. 
\end{proof}

Using Lemma \ref{lem-finvariant} we can also show:

\begin{lema}\label{lem-invdeLambda}
For each closed $G_t$-invariant set $K\subset \Lambda_f$ we can define a closed $f$-invariant set $K_f \subset M$, with the property that if $K,K'\subset\Lambda_f$ are nonempty and disjoint closed $G_t$-invariant sets, then $K_f$ and $K'_f$ are nonempty and disjoint.  
\end{lema} 
\begin{proof}

It is best to work in $T^1\mt$ and $\mt$. Lift $K$ to $\widetilde K \subset T^1\mt$ and consider the set $A_{\widetilde{K}}$ of pairs $(a_-, a_+) \in \partial \HH^3 \times \partial \HH^3$ such that there is a geodesic in $\widetilde K$ whose backward limit is $a_-$ and the forward limit is $a_+$. Now, we consider the set:

$$ \widetilde{K_f} = \{ x \in \mt \ : \ \exists (a_-,a_+) \in A_{\widetilde{K}} \ : e^-(x)=a_- \ , \ e^+(x)=a_+ \}. $$

This set is non empty if $K$ is non empty and it is $\pi_1(M)$-invariant and $\widetilde f$-invariant thanks to Lemma \ref{lem-finvariant}. We denote $K_f$ its projection to $M$, we need to show that $K_f$ is compact. Notice that it is clear by its definition that if $K \cap K' = \emptyset$ then $K_f \cap K_f' = \emptyset$. 

To show that $K_f$ is compact let us show that $\widetilde{K_f}$ is closed. Note that since $K$ is closed, so is $\widetilde{K}$, and this implies that $A_{\widetilde{K}}$ is closed in $(\partial \HH^3 \times \partial \HH^3) \setminus \Delta$ where $\Delta = \{(\xi,\xi) \ : \ \xi \in \partial \HH^3\}$. If $x_k$ is a sequence in $\widetilde{K_f}$ coverging to some point $x_\infty$ in $\widetilde{M}$, it follows by Lemma \ref{lem-emascont} that the sequence $(e^+(x_k),e^-(x_k))$ converges to $(e^+(x_\infty),e^-(x_\infty))$. Since $(e^+(x_k),e^-(x_k))$ is a sequence in $A_{\widetilde{K}}$ and $A_{\widetilde{K}}$ is closed,  it follows that $(e^+(x_\infty),e^-(x_\infty))$ is a point in $A_{\widetilde{K}}$. Then $x_\infty$ is a point in $\widetilde{K_f}$.

\end{proof}

Thus, we have reduced Theorem \ref{teo-QGhomeo} to proving things about the set $\Lambda_f$ for the geodesic flow. The fact that the geodesic flow is Anosov and that $\Lambda_f$ is not \emph{transversally totally disconected} (which we shall precise in the next section) will give us the desired results (see \cite{AR,KeynesSears} for a general proof assuming only expansiveness). 

\subsection{Invariant subsets} 

We will first show that the set $\Lambda_f$ contains a non-trivial connected set in the weak unstable manifold transverse to the flow lines. We continue with the notation of the previous subsection. We first show: 

\begin{lema}\label{lem-nontrivialarc}
There exists an arc $\eta_0: [0,1] \to \mt$ with the property that $e^{+}(\eta_0(t))$ is non constant. 
\end{lema}

\begin{proof}
This is just the fact that $e^{+}$ is continuous and $\pi_1(M)$-equivariant, thus non constant (cf. Corollary \ref{cor-nonconst}). So, we can consider $\eta_0$ to be a continuous curve joining two points with different image by $e^{+}$. 
\end{proof}

\begin{remark}
The previous lemma states that the set $\Lambda_f$ has a connected set which is not completely contained in a weak stable manifold of the geodesic flow. The next lemma will prove the classical fact that this implies that it must contain a connected set inside some weak unstable manifold. 
\end{remark}

Using the dynamics of $\pi_1(M)$, we will show that we can find a continuum of $\Lambda_f$ contained in a weak unstable manifold and not contained in a flowline. Recall that a \emph{chainable continuum} is a continuum (i.e. compact and connected set) such that for every $\eps>0$ there is a finite open cover $\{O_i\}_{i=1}^{n}$ with sets of diameter $\leq \eps$ such that each $O_i$ has non empty intersection with $O_{i-1}$ and $O_{i+1}$ only. Every Hausdorff limit of arcs of bounded diameter in $T^1\mt$ contains a non-trivial chainable continuum. We shall show:

\begin{lema}\label{lem-strongunstable}
There is $\xi \in \partial \HH^3$ and a non trivial chainable continuum $\cC \subset \tilde \Lambda_f$ such that for every $(p,v) \in \cC$ one has that $v_- = \xi$ (recall equation \eqref{eq:lambdaf}). 
\end{lema}

\begin{proof}
Let $\eta_0$ be the curve constructed in the previous lemma. Consider the geodesic $c$ in $\HH^3$ joining the points $v_-= e^{-}(\eta_0(0))$ and $v_+=e^{+}(\eta_0(0))$. Take a sequence $x_n$ in $c$ converging to $v_+$ and deck transformations $\gamma_n \in \pi_1(M)$ so that $\gamma_n x_n$ belongs to a given compact fundamental domain of $M$ in $\mt$. 

We idenfity $\HH^3$ with the unit ball in $\RR^3$ and $\partial \HH^3$ with  $S^2$ the unit sphere. This way, we can talk about distances in $S^2$ with the induced metric of $\RR^3$. 

Since $\gamma_n \to \infty$ in $\pi_1(M)$ (this is equivalent to saying that the norm $\gamma_n \in \mathrm{Isom}(\HH^3) \cong \mathrm{PSL}_2(\CC)$ goes to infinity) we know that there is a sequence of neighborhoods $U_n \subset \partial \HH^3$ shrinking to $v_+$ (with the induced metric of $\RR^3$) whose complement is mapped in some open set $V_n$ of diameter going to $0$ (again, with the induced metric of $\RR^3$). We can choose also $U_n$ and $V_n$ so that the action of $\gamma_n$ expands all vectors tangent to $U_n$. Also, since $v_+ \neq v_-$, for large enough $n$ we know that $v_- \notin U_n$. By cutting $\eta_0$ if necessary, we can assume that $e^{-}(\eta_0([0,1]) \cap U_n = \emptyset$. Up to taking a subsequence, we know that $V_n \to \xi  \in \partial \HH^3$ and thus we get that $\gamma_n e^{-}(\eta_0([0,1]))= e^{-}(\gamma_n \eta_0([0,1])) \to \xi $. 



For every sufficiently large $n$ we can choose $t_n$ so that the diameter of the set $e^+(\gamma_n \eta_0 ([0,t_n]))$ in $\partial \HH^3$ (with the induced metric from $\RR^3$) is exactly $1$. For every $t\in [0,t_n]$ let $c_n(t)\subset \widetilde{\Lambda_f}$ be the geodesic joining $e^+(\gamma_n\eta_0(t))$ with $e^-(\gamma_n\eta_0(t))$. There exists $D\subset T^1 \HH^3 $ a fixed compact fundamental domain of $T^1M$ so that for every $n$ large enough there exists $\alpha_n:[0,1]\to D$ continuous such that $\alpha_n(t)\in c_n(t_nt)$ for every $t\in [0,1]$. Taking limit (with $n$) of the arcs $\alpha_n$ gives the desired set $\cC$.



\end{proof}

We are now in conditions to prove Theorem \ref{teo-QGhomeo}: 

\begin{proof}[Proof of  Theorem \ref{teo-QGhomeo}] 

Let $\cC \subset \widetilde{ \Lambda_f}$ be the chainable continuum in Lemma \ref{lem-strongunstable}. Recall that by the construction of $\cC $, it holds that $v_-= w_-$ for all $(p,v), (q,w) \in \cC$ so we get that the set $\cC$ is a chainable continuum completely contained in a weak unstable manifold and not contained in a single orbit.

Now, consider $I \subset \Lambda_f$ to be the projection of $\cC$ to $T^1M$. We note first that we can assume that $I$ is contained in a strong unstable manifold since $\Lambda_f$ is saturated by flowlines.

We will use the following simple property whose proof we can omit.  

\begin{claim}\label{claim-duplicate1}
There exists $T>0$ and $\epsilon_0>0$ such that for every $\epsilon \in (0,\epsilon_0)$, if $J\subset T^1M$ is a compact connected set of diameter $\epsilon$ contained in a strong unstable manifold, then for every $t\geq T$ the set $G_t(J)$ has diameter larger than $2\epsilon$.
\end{claim}

The following uses an idea of Ma\~n\'e (\cite{Manhe}, see also \cite{AR,KeynesSears}):

\begin{claim}\label{claim-duplicate2}

Given a finite number of points $\{x_1,\ldots,x_N\}\subset \Lambda_f$ with pairwise disjoint $G_t$-orbits, there exists $K\subset \Lambda_f$ compact and $G_t$-invariant such that $K\cap \{x_1,\ldots,x_N\}=\emptyset$.
\end{claim}

\begin{proof}

Let $T>0$ and $\epsilon_0>0$ be as in Claim \ref{claim-duplicate1}. Let $\mathcal{W}^s$ and $\mathcal{W}^u$ denote the strong stable and unstable foliations, respectively, of the geodesic flow $G_t$. For every $r>0$ and $x\in T^1M$ let $\mathcal{W}^s_r(x)$ and $\mathcal{W}^u_r(x)$ denote the ball of center $x$ and radius $r$ in the strong stable and unstable leaf through $x$, respectively.

Let $D_i:=\mathcal{W}^u_{10\epsilon}(\mathcal{W}^s_{10\epsilon}(x_i))$ for every $i\in \{1,\ldots,N\}$, for some $\epsilon\in (0,\epsilon_0)$ small enough so that if $i\neq j$, then $D_i\cap D_j=\emptyset$ and,  moreover, if a point $x\in D_i$ satisfies that $G_t(x)\in D_j$ for some $t>0$, then $t\geq T$. Moreover, let $\epsilon\in (0,\epsilon_0)$ be small enough so that $I$ admits a sub-chainable continua $J_1$ of diameter $\epsilon$. Also let $D'_i:=\mathcal{W}^u_{\epsilon/2}(\mathcal{W}^s_{\epsilon/2}(x_i))$ for every $i\in \{1,\ldots,N\}$.

Note that $\diam(G_t(J_1))$ tends to infinity with $t$ since $J_1$ is contained in a strong unstable manifold of $G_t$. Let $t_1>0$ be the first positive time such that one of the following two conditions happen: either $\diam(G_{t_1}(J_1))=3\epsilon$ or $G_{t_1}(J_1)\cap D'_{i_1}\neq \emptyset$ for some ${i_1}\in \{1,\ldots,N\}$. If the former happens, let $J_2\subset G_{t_1}(J_1)$ be a sub-chainable continua such that $\diam(J_2)=\epsilon$. If the latter happens (i.e. $G_{t_1}(J_1) \cap D'_{i_1}$ while $\mathrm{diam}(G_{t_1}(J_1)) \leq 3\eps$), note that $G_{t_1}(J_1)\subset D_{i_1}$, and let $J_2\subset G_{t_1}(J_1)$ be a sub-chainable continua disjoint from $D_{i_1}'$ such that $\diam(J_2)=\epsilon$. Let $I_1=J_1$ and $I_2:=G_{-t_1}(J_2)$. Note that $I_2 \subset I_1$ and that for $x\in I_2$ one has that $G_t(x) \notin D'_1 \cup \ldots \cup D_N'$ for $0 \leq t \leq t_1$.

Now, let $t_2>0$ be the first positive time such that, again, one of the following two conditions happen: either $\diam(G_{t_2}(J_2))=3\epsilon$ or $G_{t_2}(J_2)\cap D'_{i_2}\neq \emptyset$ for some ${i_2}\in \{1,\ldots,N\}$. Note that $t_2 \geq T$. If the former happens, let $J_3\subset G_{t_2}(J_2)$ be a sub-chainable continua such that $\diam(J_3)=\epsilon$. If the latter happens, note that $G_{t_2}(J_2)\subset D_{i_2}$, and let $J_3\subset G_{t_2}(J_2)$ be a sub-chainable continua disjoint from $D'_{i_2}$ such that $\diam(J_3)=\epsilon$ (note that in this case, this is possible since $\diam(G_{t_2}(J_2)) \geq 2\epsilon$ is guaranteeed because of Claim \ref{claim-duplicate1} and the fact that $t_2 \geq T$). Let $I_3:=G_{-t_1-t_2}(J_3)$. Note that $I_3\subset I_2 \subset I_1$ and that for $x\in I_3$ one has that $G_t(x) \notin D'_1 \cup \ldots \cup D_N'$ for $0 \leq t \leq t_1+t_2$..

Inductively, one constructs $I_1 \supset I_2 \supset \ldots$ a decreasing sequence of non empty sub-chainable compact set such that any point $x\in I_n$ satisfies that from time 0 to time $t_1+\ldots+t_n \geq (n-1)T$ the $G_t$-orbit of $x$ does not intersect $D'_1\cup\ldots \cup D'_N$. Since $I_n$ are nested non empty compact sets, one knows that $\bigcap_n I_n \neq \emptyset$. It follows that the positive orbit of every $x\in \bigcap_n I_n$ does not intersect $D'_1\cup\ldots \cup D'_N$. Then it suffices to take $K$ equal to the omega limit of a point in $\bigcap_n I_n$ to obtain a compact $G_t$-invariant subset of $\Lambda_f$ disjoint from $\{x_1,\ldots,x_N\}$. This ends the proof of the claim.
\end{proof}

To finish the proof of Theorem \ref{teo-QGhomeo} we can now argue as follows. Let $\Lambda_1\subset \Lambda_f$ be a $G_t$-minimal set and $x_1$ be a point in $\Lambda_1$. By Claim \ref{claim-duplicate2} there exists $K_2\subset \Lambda_f$ a compact $G_t$-invariant set such that $K_2\cap \{x_1\}=\emptyset$. Let $\Lambda_2\subset K_2$ be a $G_t$-minimal set. Note that $\Lambda_1$ and $\Lambda_2$ compact minimal sets, thus, they must be disjoint (becase minimal sets are disjoint or equal and $x_1 \in \Lambda_1 \setminus \Lambda_2$).

Inductively, if $\Lambda_1,\ldots, \Lambda_N$ are disjoint $G_t$-minimal subsets of $\Lambda_f$, then, by taking $x_i$ a point in $\Lambda_i$ for every $i\in \{1,\ldots,N\}$, there exists by Claim \ref{claim-duplicate2} a compact $G_t$-invariant set $K_{N+1}\subset \Lambda_f$ disjoint from $\{x_1,\ldots,x_N\}$. By a similar argument as above, if $\Lambda_{N+1}\subset K_{N+1}$ is a minimal $G_t$-invariant set, then $\Lambda_{N+1}$ is disjoint from every $\Lambda_i$ in $\{\Lambda_1,\ldots,\Lambda_N\}$. 

This way, inductively, one can contruct an infinite number of pairwise disjoint compact $G_t$-minimal subsets of $\Lambda_f$. By Lemma \ref{lem-invdeLambda} these sets correspond to pairwise disjoint $f$-invariant compact subset of $M$.
\end{proof}

\subsection{Topological entropy}\label{ss.entropy}

Recall that a homeomorphism $h: M \to M$ has positive topological entropy if there is $\eps>0$ and a constant $s>0$ so that for every large enough $n$ there is a set $F_n$ with more than $e^{sn}$ elements such that if $x,y \in F_n$ are distinct points, then there is some $1\leq i \leq n$ such that $d(h^i(x),h^i(y))>\eps$. We refer the reader to \cite{KH} for more on topological entropy. 

Here we show the following proposition which completes the proof of Theorem \ref{teoA}.   

\begin{prop}
Let $f: M \to M$ be a quasi-geodesic homeomorphism of a closed hyperbolic 3-manifold $M$. Then, $f$ has positive topological entropy. 
\end{prop}

\begin{proof} 
The geodesic flow $\widetilde{G_t}$ in $T^1 \widetilde{M}$ is uniformly expanding to the future in the following sense: For every $K>0$ and $\delta>0$ there exists $T>0$ such that if $x,y$ are points in the same strong unstable leaf such that $d(x,y)>\delta$, then $d(\widetilde{G_t}(x),\widetilde{G_t}(y))>K$ for every $t\geq T$. 

Consider a small compact connected set $\cC$ in $\mt$ of diameter $\eps>0$ whose projection to $\widetilde{\Lambda_f}$ contains a connected set of positive diameter $\delta>0$ within a strong unstable leaf (see Lemma \ref{lem-strongunstable}). Using the uniform expansion of $\widetilde{G_t}$ in $T^1 \widetilde{M}$ and the fact that orbits of $\ft$ uniformly shadow the orbits of the geodesic flow, it follows that there is some $N>0$ (depending only on $\delta$) so that $\cC$ contains two compact connected subsets $\cC_1, \cC_2 \subset \cC$ so that $d(f^N(\cC_1),f^N(\cC_2))>\epsilon$ and the projections of $f^N(\cC_1)$ and $f^N(\cC_2)$ to $\widetilde{\Lambda_f}$ each contain a connected set of positive diameter $\delta>0$ within a strong unstable leaf.  Repeating this procedure, we get that we can find at least $2^k$ points in $\cC$ whose orbits separate more than $\eps$ in some iterate between $0$ and $kN$. This shows that the entropy of $f$ is at least $\frac{1}{N} \log(2)>0$.


\end{proof}

We close this section by noticing that while we have exploited the sort of \emph{semiconjugacy} given by the maps $e^{\pm}$ to produce information about the homeomorphism, there are still many questions which are unclear. Even in the flow case (see \cite{FrLa}) one can produce examples which differ greatly from the models (for instance, one can blow up orbits to eliminate transitivity, or even create new dynamics by perturbation). For flows, in the recent \cite{FrLa} it has been obtained that one can associate a natural `core' dynamics to each flow, and that this core dynamics contains a \emph{pseudo-Anosov flow}. In our case, we are far from obtaining something similar, except under the stronger assumptions of Theorem \ref{teoC} where we manage to obtain a similar result.


\section{Proof of Theorem \ref{teoB}}\label{ss.teobandex}

To prove Theorem \ref{teoB} we need to express the problem in terms of linear cocycles in order to be able to apply the following result from \cite{BPS} (their result is stronger than what we state): 

\begin{teo}[Theorem 4.12 \cite{BPS}]\label{teoBPS}
Let $T: X \to X$ be a minimal homeomorphism and $\cA: X \to \mathrm{PSL}_2(\CC)$ be\footnote{In \cite{BPS} they work with real matrices, but complex matrices can be thought as inside real matrices in the double of the dimension. Similarly, the fact that we quotient by $\pm \mathrm{id}$ is not an issue since the norm is still well defined.} a continuous function and denote by $\cA^{(n)}(x) = \cA(T^{n-1}x) \cdots \cA(x)$. Then either the set of points on which $\liminf_{n} \frac{1}{n} \log \|\cA^{(n)}(x)\| =0$ is $G_\delta$-dense, or, there are constants $C>0, \tau>0$ such that for every $x \in X$ and $n>0$ one has that $$\|\cA^{(n)}(x)\|> C e^{\tau n}.$$
\end{teo}

Note that the value of $\liminf_n \frac{1}{n} \|\cA^{(n)}(x)\|$ (which equals $\tau$) is independent on the chosen norm in $\CC^2$, while the value of $C$ might depend on it. 

We can translate this into our context as follows: 

\begin{coro}\label{coroBPS}
Let $f: M \to M$ be a minimal homeomorphism of a closed hyperbolic 3-manifold and $\ft$ be a good lift. Then, either there is a $G_\delta$-dense subset of $M$ of points $x$ such that if $\tilde x \in \mt$ proyects to $x$ and such that $\liminf_n \frac{1}{n} d(\ft^n(\tilde x), \tilde x) =0$,  or, $f$ is a quasi-geodesic homeomorphism. 
\end{coro}

This corollary implies Theorem \ref{teoB}, since we have from Theorem \ref{teoA} that a minimal homeomorphism cannot be quasi-geodesic, so the first option must hold. 

\begin{proof}[Proof of Corollary \ref{coroBPS}]
The proof in \cite{BPS} adapts directly to general subadditive sequences, but we will instead show that in this case we can put ourselves in the same conditions as in Theorem \ref{teoBPS}. 

First choose a trivialization of the frame bundle of $M$ (this can always be achieved up to finite cover, since every orientable 3-manifold is paralelizable, see \cite{BL}) and consider a continuous choice of frame at each point of $M$ which lifts to a framing of  $T\mt$. This way, for each $x \in \mt$ we can find a unique isometry $g_x$ sending $x$ to $\ft(x)$ and respecting the chosen framing. This way, $g_x \in \mathrm{Isom}(\HH^3) \cong \mathrm{PSL}_2(\CC)$ is a continuous choice of matrices and it verifies that $g_x = g_{\gamma x}$ for every $\gamma \in \pi_1(M)$ because $\ft(\gamma x)= \gamma \ft(x)$ and $\gamma$ respects the framing. 

This allows one to define a linear cocycle as $\cA: M \to \mathrm{PSL}_2(\CC)$ where $\cA(x) = g_{\tilde x}$ where $\tilde x \in \mt$ projects to $x$.  Note that given $x \in \HH^3$, there exists a norm $\|\cdot \|$ so that $d(\ft^n(\tilde x), \tilde x)=2 \log \|\cA^{(n)}(x)\|$ and therefore the dichotomy given in Theorem \ref{teoBPS} translates directly in the dichotomy claimed in Corollary \ref{coroBPS}.
\end{proof}

In \cite{BPS} there is an example showing that Theorem \ref{teoBPS} requires minimality (see \cite[Example 3.12]{BPS}). Here we give an example in our setting to show that cannot remove the minimality assumption either.  It provides an example where the escape rate is positive for every orbit, but it is not quasi-geodesic. 

\begin{example}
Let $\varphi_0 : S \to S$ be a pseudo-Anosov homeomorphism of a closed surface of genus $g\geq 2$ and assume it has a regular fixed point $p$. One can blow up $p$ so that $\varphi_0$ has a neighborhood of fixed points containing $p$ and is $C^0$-close to $\varphi$ (thus it is homotopic). Call this new map $\varphi_1$ and we consider $\phi^1_t : M \to M$ be the suspension flow on $M = S \times [0,1]/_{(x,1) \sim (\varphi_1(x),0)}$ which is a hyperbolic 3-manifold \cite{Thurston}. The flow $\phi^1_t: M \to M$ is quasi-geodesic \cite{Zeghib}. We will modify the flow in a neighborhood of the solid torus obtained by suspending the neighborhood of $p$ made of fixed points as in Figure \ref{example_flow}. This way, we produce a new flow $\phi^2_t: M \to M$ on which every ray is quasi-geodesic with the same escape rate, but the time needed to see the escape rate goes to infinity, and in particular, there are full orbits which are not quasi-geodesic. 
\end{example}

\begin{figure}[ht]
    \centering
    \def\svgwidth{\columnwidth}
    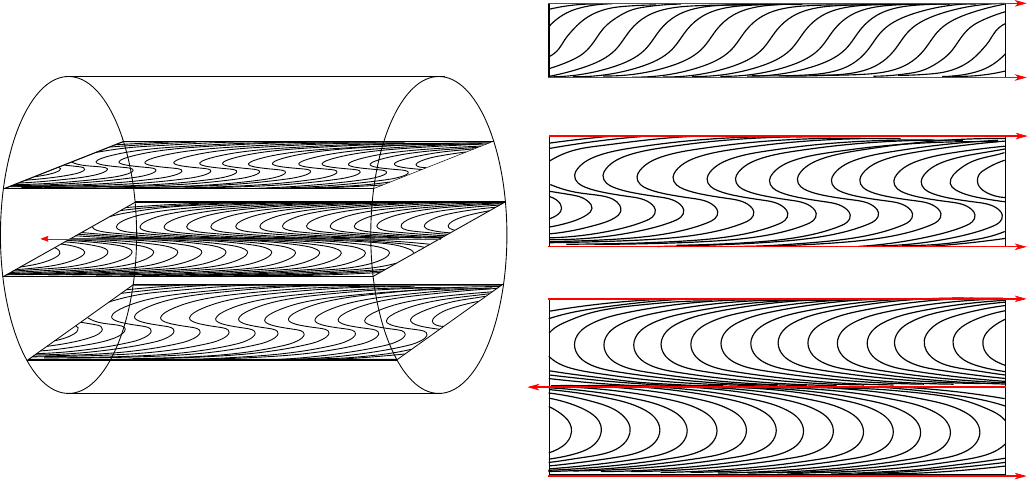
    \caption{\small{The figure depicts a deformation of the flow in a neighborhood fully consisting of periodic orbits of same period. After the deformation, in the middle section one gets orbits going in both directions.}}
    \label{example_flow}        
\end{figure}

Theorem \ref{teoB} states that under minimality assumptions, positive scape rate to infinity is impeded on a residual set of points. Further can be said regarding orbits with a quasi-geodesic behaviour.

Given a homeomorphism $f: M \to M$ of a closed hyperbolic 3-manifold and $\ft$ a good lift, we say that a point $x\in M$ has a $\lambda$-\emph{quasi-geodesic orbit} if there exists $\lambda>0$ such that $\lambda^{-1} |n-m| - \lambda \leq d(\ft^n(\tilde{x}), \ft^m(\tilde{x})) \leq \lambda|n-m| + \lambda$ for every $n,m \in \mathbb{Z}$ and $\tilde{x}$ lift of $x$ (i.e. the map $\mathbb{Z} \to \widetilde{M}=\mathbb{H}^3$ given by $n\mapsto \tilde{f}^n(\tilde{x})$ is a quasi-isometry). 

It is immediate to check that for every $\lambda>0$ the set $K_\lambda\subset M$ of points having a $\lambda$-quasi-geodesic orbit is an $f$-invariant and closed subset of $M$. Since Theorem \ref{teoA} states that for every $\lambda>0$ the set $K_\lambda$ is a proper subset of $M$, as a corollary we get:

\begin{coro}
Let $f: M \to M$ be a minimal homeomorphism of a closed hyperbolic 3-manifold. Then $f$ has no quasi-geodesic orbit.
\end{coro}


\section{Positive escape rate against foliations}\label{s.positivefol}
 
 In this section we prove Theorem \ref{teoC}. The results in this section were obtained in \cite{Gomes} and we follow its presentation. So, here $f: M \to M$ will be a homeomorphism of a hyperbolic 3-manifold and $\ft$ a good lift. We will assume that $f$ has positive escape rate with respect to a uniform $\RR$-covered foliation $\cF$ as in Definition \ref{defi positive rate w.r.t. foliation}. The goal is to prove Theorem \ref{teo.B1} which follows from Propositions \ref{prop exists T sub gamma} and \ref{prop if eta and gamma different, T projections are disjoint} as explained in \S \ref{ss.RcovI}. We will use the notations introduced in \S \ref{ss.RcovI}.

\subsection{The quasi-geodesic property}
Here we prove Proposition \ref{prop.qg}. Recall that $Z$ denotes the structure map of the foliation $\cF$ (cf. Proposition \ref{p.existeZ}).  We note that what we prove here can also be deduced directly from Proposition \ref{prop-QM} (shown in the next section) by choosing a parametrization of the leaf space using the structure map in order to produce a function into the leaf space satisfying \eqref{eq:QM}. 

Suppose $f: M \to M$ is a homeomorphism homotopic to the identity and $\widetilde{f}$ is its good lift to $\widetilde{M}$. We write $\widetilde{f}^k > Z$ to mean that for every $L \in \F$ and every $x \in L$, the leaf through $\widetilde{f}^k(x)$ is above $Z(L)$. We will denote the leaf $L \in \F$ containing a point $y \in \mt$ as $L(y)$. 

\begin{lemma}
\label{lema existe k tq f a la k pasa Z}
If $f : M \to M$ has positive escape rate with respect to $\cF$, then there exists $k \in \mathbb{N}$ such that $\widetilde{f}^k > Z$.
\end{lemma}

\begin{proof}
First, we will see that it is sufficient to find \( K \geq 1 \) such that for every \( x \in \widetilde{M} \), there exists \( k_x \in \{ 1, \ldots, K \} \) such that \( L(\widetilde{f}^{k_x}(x)) \geq Z(L(x)) \). Suppose such a \( K \) exists. Since \( \widetilde{f} \) is at a bounded distance from the identity, there exists \( m \in \mathbb{Z}_{\geq 0} \) such that \( L(\widetilde{f}^i(x)) \geq Z^{-m}(L(x)) \) for all \( x \in \widetilde{M} \) and \( i \in \{ 1, \ldots, K \} \), given that the distance between every leaf and its image under \( Z \) is uniformly far from zero. Considering $k> (m+1)K$ it follows that $\widetilde{f}^k > Z$.

Now let us show that such a \( K \) exists. If not, there would be a sequence \( x_n \in \widetilde{M} \) such that \( L(\widetilde{f}^j(x_n)) < Z(L(x_n)) \) for all \( j \in \{ 1, \ldots, n \} \). Due to the compactness of \( M \), up to taking a subsequence, there exist elements \( \gamma_n \in \pi_1(M) \) such that \( \gamma_n x_n \) converges to a point \( x \in \widetilde{M} \). Since \( \widetilde{f} \) and \( Z \) commute with the action of \( \pi_1(M) \) and preserve the orientation of the leaf space, we still have \( L(\widetilde{f}^j(\gamma_n x_n)) < Z(L(\gamma_n x_n)) \) for all \( n \) and \( j \in \{ 1, \ldots, n \} \). By hypothesis, we know that the orbit of \( x \) tends to infinity in the leaf space, so there exists \( j \) such that \( L(\widetilde{f}^j(x)) > Z^2(L(x)) \). Let \( N \) be such that for all \( n \geq N \) we have \( Z(L(\gamma_n x_n)) < Z^2(L) \) (which can be ensured by the continuity of \( Z \)), and also \( L(\widetilde{f}^j(\gamma_n x_n)) > Z^2(L(x)) \) (this is possible by the continuity of \( \widetilde{f}^j \)). In particular, when \( n \geq N \), we would have \( L(\widetilde{f}^j(\gamma_n x_n)) \geq Z(L(\gamma_n x_n)) \), leading to a contradiction.
\end{proof}

As a consequence, since the distance between any leaf \( L \in \F \) and its iterates \( Z^n(L) \) tends to infinity with \( n \), it follows that \( f \) escapes to infinity with positive speed. This justifies the name ``escape rate" with respect to \( \cF \). It is also enough to deduce Proposition \ref{prop.qg}.  

\begin{coro}
\label{coro escape de compactos}
If \( f: M \to M \) has positive escape speed with respect to \( \cF \), then \( f \) is quasi-geodesic. 
\end{coro}
\begin{proof}
This follows from the fact that there is a uniform lower bound in the distance between $L$ and $Z^n(L)$ of the order of $cn$ for some positive $c$. This gives the uniform lower bound in the quasi-geodesic definition, the upper bound being automatic from the fact that $\ft$ is at bounded distance from the identity. 
\end{proof}

\subsection{Some properties of the pseudo-Anosov regulating flow}\label{ss.PA} 

This section introduces the most important ingredient in the proof of Theorem \ref{teo.B1}: the \textit{regulating pseudo-Anosov flow} of a foliation, cf. Theorem \ref{thm-regulating}.

\begin{defi}
\label{defi topological pA flow}
    A flow \( \phi_t: M \to M \) is a \textit{topological pseudo-Anosov flow} if it preserves a pair of transverse singular foliations, \( \mathcal{W}^{ws} \) and \( \mathcal{W}^{wu} \), such that
    \begin{enumerate}
        \item Every pair of orbits in the same leaf of \( \mathcal{W}^{ws} \) are future asymptotic, and every pair of orbits in the same leaf of \( \mathcal{W}^{wu} \) are past asymptotic,
        \item The singular leaves of \( \mathcal{W}^{ws} \) and \( \mathcal{W}^{wu} \) are finite and of \textit{p-prong} type along a periodic orbit of \( \phi_t \) (see Figure \ref{pA_flow}).
    \end{enumerate}

     The orbits of \( \phi_t \) without prongs are called \textit{regular}. Similarly, \textit{regular points} are those whose orbit is regular.
\end{defi} 

\begin{figure}[h!]
    \centering
    \def\svgwidth{0.9\columnwidth}
    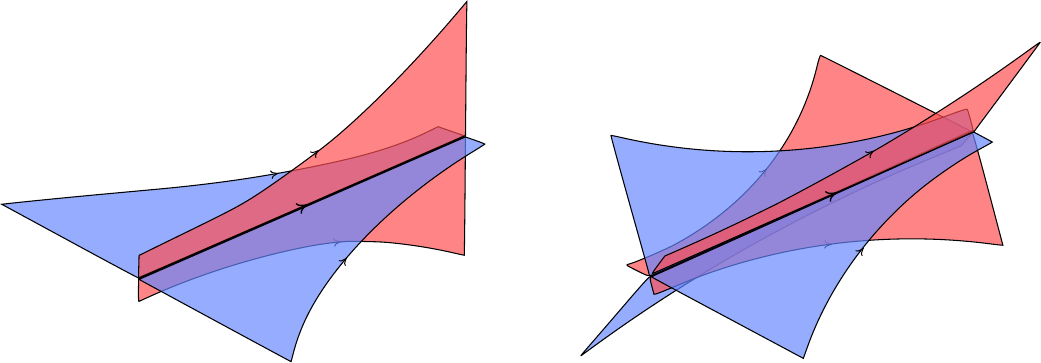
    \caption{A regular orbit of a pseudo-Anosov flow on the left, and a 3-pronged singular orbit on the right.}
    \label{pA_flow}        
\end{figure}

An important property of pseudo-Anosov flows is the existence of periodic orbits. Note that since there are finitely many periodic orbits of $p$-prong type with $p\geq 3$, we get from the following that there exist infinitely many distinct regular periodic orbits. 

\begin{prop}
\label{prop periodic orbits}
    If \( \phi_t: M \to M \) is a topological pseudo-Anosov flow in a closed hyperbolic 3-manifold, then \( \phi_t \) has infinitely many non freely homotopic periodic orbits. Moreover, the flow is transitive and every orbit is approximated by periodic orbits in the sense that given $x \in M$, $T>0$ and $\eps>0$, there is $y \in M$ and some increasing homeomorphism $h:\RR \to \RR$ so that the $\phi_t$-orbit of $y$ is periodic and such that $d(\phi_t(x), \phi_{h(t)}(y))< \eps$ for all $0 \leq t \leq T$. 
\end{prop}

\begin{proof}
Note that if a pseudo-Anosov flow is non-transitive, then, it is transverse to an incompressible torus or Klein bottle (see \cite{Mosher}) thus, since $M$ is hyperbolic, we can assume it is transitive. Since pseudo-Anosov flows admit Markov partitions (see \cite{Ia} for a very general statement) we deduce that the flow has infinitely many periodic orbits, which cannot be all freely homotopic to each other (see \cite{Mosher}). The approximation of periodic orbits is a classical consequence of the existence of Markov partitions (see \cite{Ia}). 

\end{proof}

If \( \phi_t \) is a topological pseudo-Anosov flow transverse to \( \F \), on each leaf \( L \in \F \) there exist \( \mathcal{G}^s \) and \( \mathcal{G}^u \) transverse singular foliations of dimension \( 1 \), given by the intersections of \( \mathcal{W}^{ws} \) and \( \mathcal{W}^{wu} \) with \( L \), which vary continuously with \( L \).

We will call a \textit{stable line} an embedding of \( \R \) into a leaf of \( \mathcal{G}^s \). That is, if \( \mathcal{G}^s(x) \) is a leaf of \( \mathcal{G}^s \) that does not contain any prongs, then the only line contained in it is \( \mathcal{G}^s(x) \) itself. If \( \mathcal{G}^s(x) \) contains a prong, however, there will be several lines contained in \( \mathcal{G}^s(x) \), corresponding to taking each of the possible ``paths" who reach the prong (see Figure \ref{prong and line}).

\begin{figure}[ht]
    \centering
    \def\svgwidth{\columnwidth}
\begingroup%
  \makeatletter%
  \providecommand\color[2][]{%
    \errmessage{(Inkscape) Color is used for the text in Inkscape, but the package 'color.sty' is not loaded}%
    \renewcommand\color[2][]{}%
  }%
  \providecommand\transparent[1]{%
    \errmessage{(Inkscape) Transparency is used (non-zero) for the text in Inkscape, but the package 'transparent.sty' is not loaded}%
    \renewcommand\transparent[1]{}%
  }%
  \providecommand\rotatebox[2]{#2}%
  \newcommand*\fsize{\dimexpr\f@size pt\relax}%
  \newcommand*\lineheight[1]{\fontsize{\fsize}{#1\fsize}\selectfont}%
  \ifx\svgwidth\undefined%
    \setlength{\unitlength}{580.99980099bp}%
    \ifx\svgscale\undefined%
      \relax%
    \else%
      \setlength{\unitlength}{\unitlength * \real{\svgscale}}%
    \fi%
  \else%
    \setlength{\unitlength}{\svgwidth}%
  \fi%
  \global\let\svgwidth\undefined%
  \global\let\svgscale\undefined%
  \makeatother%
  \begin{picture}(1,0.36524223)%
    \lineheight{1}%
    \setlength\tabcolsep{0pt}%
    \put(0,0){\includegraphics[width=\unitlength,page=1]{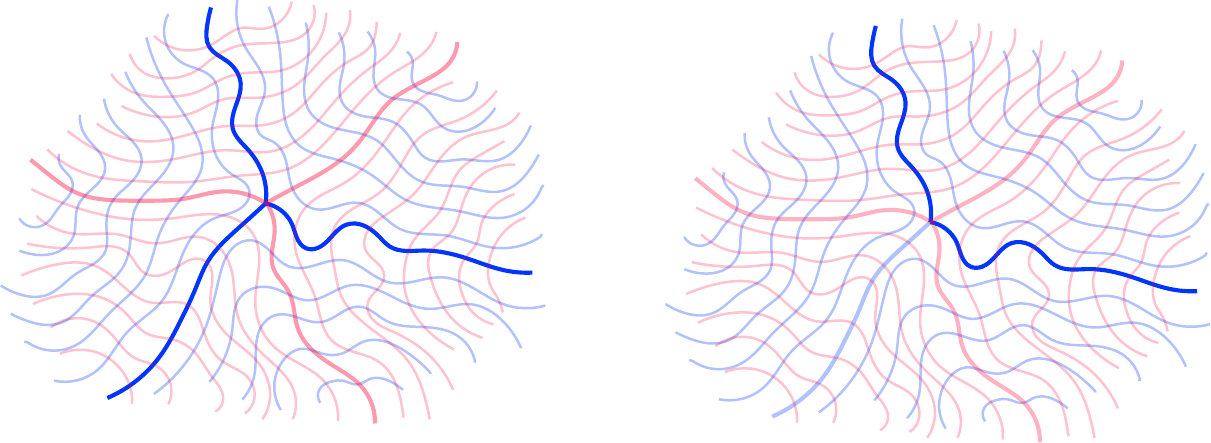}}%
    \put(0.01527776,0.03739287){\color[rgb]{0,0.20784314,1}\makebox(0,0)[lt]{\lineheight{10.25}\smash{\begin{tabular}[t]{l}$\mathcal G^s$\end{tabular}}}}%
    \put(0.00663137,0.25643461){\color[rgb]{1,0,0.20784314}\transparent{0.80000001}\makebox(0,0)[lt]{\lineheight{10.25}\smash{\begin{tabular}[t]{l}$\mathcal G^u$\end{tabular}}}}%
  \end{picture}%
\endgroup%

    \caption{A stable line through a $3$-prong.}
    \label{prong and line}        
\end{figure}

Recall that, as a corollary of Candel's Theorem, the leaves of \( \F \) with the metric induced by \( \mathbb{H}^3 \) are quasi-isometric to \( \mathbb{H}^2 \) via a quasi-isometry whose constant is independent of the leaf (Theorem \ref{teo 'Candel'}).  


The construction in \cite[Thm. 9.31]{Calegari-book} (see also \cite{FenleyRcov}) indeed shows the following.

\begin{fact}
\label{fact pA flow with all the features}
    Let \( \cF \) be a uniform \( \R \)-covered foliation of a hyperbolic 3-manifold \( M \). The pseudo-Anosov flow given by Theorem \ref{thm-regulating} can be considered with the following properties.

    \begin{enumerate}
        \item \label{factpAfeatures(i)} \textbf{Regularity and transversality.} The orbits of the flow are \( C^1 \) curves, transverse to the leaves of \( \F \). The singular foliations \( \mathcal{W}^{ws} \) and \( \mathcal{W}^{wu} \) are \( C^{0, 1} \).
        
        \item  \label{factpAfeatures(ii)} \textbf{Quasi-geodesic lines.} There exists \( C \geq 1 \) such that every arc-length parametrization \( \{ \ell(s) \}_{s \in \R} \) of a stable or unstable line in a leaf \( L \in \widetilde{F} \) is a \( C \)-quasi-isometric embedding of \( \R \) into \( L \). That is, for all \( t, s \in \R \), we have \[ C^{-1} d_L(\ell(t), \ell(s)) - C \leq |t-s| \leq  C d_L(\ell(t), \ell(s)) + C.\]
        
        \item \label{factpAfeatures(iii)} \textbf{Bounded intersection angle.} There exists \( \theta_0 \in (0, \pi/2] \) such that if $\ell^s$ is a stable line and $\ell^u$ is an unstable line both in a leaf $L \in \F$ which we identify with $\HH^2$ via a uniform quasi-isometry, then, the angle of intersection of their geodesic representatives (cf. Proposition \ref{morselema}) is greater than \( \theta_0 \). 
            \end{enumerate}
\end{fact}

The last property can also be phrased in terms of cross ratios of cuadruples of points in the circle at infinity. The fact that the cross ratios are far from $0$ and $\infty$ is a property invariant under quasi-isometry and gives an equivalent notion. We chose this definition in order to argue in $\HH^2$ where many classical results can be quoted directly. 

Given $\cF$ a uniform \( \R \)-covered foliation of a hyperbolic 3-manifold \( M \), we will say that a flow with the properties of Fact \ref{fact pA flow with all the features} is a \textit{regulating pseudo-Anosov flow for \( \cF \)}. We define the distance \( d_{\mathcal{G}^s} \) between points \( x, y \) in the same stable line in \( L \in \F \) as the length of the line segment between them. Similarly, we define \( d_{\mathcal{G}^u} \).

\subsubsection{Contraction and expansion of lines}

Suppose \( \phi_t : M \to M \) is a regulating pseudo-Anosov flow for a uniform \( \R \)-covered foliation \( \cF \), and let \( \widetilde{\phi_t} \) and \( \F \) denote their lifts to \( \widetilde{M} \). For every pair of leaves \( L, L' \in \F \), we define the map \( \tau_{L, L'} : L \to L' \), which sends a point \( x \in L \) to the intersection of its orbit under \( \widetilde{\phi}_t \) with \( L' \). The following is stated in \cite[Fact 8.4]{BFFP} and follows from the standard properties of pseudo-Anosov flows.

\begin{prop}
\label{prop structure map contracts expands}
    There exists a constant \( \lambda > 1 \) satisfying the following. For every \( d > 0 \) there exists a natural number \( k \) such that for every pair of leaves \( L, L' \in \F \) with \( L' > Z^k(L) \), for all \( x \in L \), \( y_1 \in \mathcal{G}^u(x) \), and \( y_2 \in \mathcal{G}^s(x) \), we have
    \begin{equation*}
        \begin{aligned}
            & d_{\mathcal{G}^u}(\tau_{L, L'}(x), \tau_{L, L'}(y_1)) \geq \lambda d, \quad \text{if } d_{\mathcal{G}^u}(x, y_1) \geq d \\
            \text{and} \quad & d_{\mathcal{G}^s}(\tau_{L, L'}(x), \tau_{L, L'}(y_2)) \leq \lambda^{-1} d, \quad \text{if } d_{\mathcal{G}^s}(x, y_2) \geq d,
        \end{aligned}
    \end{equation*} where \( Z: \widetilde{M} / \F \to \widetilde{M} / \F \) is a structure map of \( \F \).
\end{prop}

Since at very small scales, more time is expected for contraction and expansion at rate \( \lambda \), for convenience we fix \( d = 1 \) and the iterate \( Z^k \) corresponding to the above proposition, ensuring contraction and expansion at rate \( \lambda \) at all scales larger than 1. Henceforth, we will refer to \( Z^k \) as \textit{the} structure map of a uniform \( \R \)-covered foliation \( \cF \). For simplicity, we denote the structure map as \( Z \) (instead of \( Z^k \)) from now on.

\subsubsection{Properties of singular foliations}

Let us assume for the remainder of this section that $\cF$ is a uniform $\R$-covered foliation and $\phi_t : M \to M$ is a pseudo-Anosov regulating flow for $\cF$. Let $\F$ and $\widetilde \phi_t$ denote their lifts to $\widetilde M$. 

\begin{fact}
\label{fact no bighorns}
    In every leaf $L \in \F$, every leaf of the singular foliation $\mathcal G^s$ of $L$ intersects each leaf of $\mathcal G^u$ \textit{at most} at a single point.
\end{fact}

The following fact is a consequence of basic hyperbolic geometry.

\begin{fact}
\label{fact angle geodesics hyperbolic plane}
    For every $\theta \in (0, \pi/2)$ and $d > 0$, there exists $d' > 0$ such that, if $\alpha, \beta_1$, and $\beta_2$ are geodesics in $\mathbb{H}^2$ such that each $\beta_i$ intersects $\alpha$ at a point $x_i$ with angle $\theta_i \in [\theta, \pi/2]$, and $d(x_1, x_2) > d'$, then $d(\beta_1, \beta_2) > d$.
\end{fact}


Let $C > 0$ such that for every $L \in  \F$, every stable or unstable line $\ell \subset L$ is contained in a $C$-neighborhood of a geodesic $g_{\ell}$ on $L$ (modulo identification of $L$ with $\mathbb H^2$).

\begin{lemma}
\label{lemma distance between intersections}
    There exists $Q > 0$ such that in every leaf $L \in \F$, for every $x \in L$, every stable line $\ell^s$ through $x$ and every unstable line $\ell^u$ through $x$, the geodesics $g_{\ell^s}$ and $g_{\ell^u}$ intersect at a point $x'$ with $d_L(x, x') < Q$.
\end{lemma}

\begin{proof}
    Let $\theta_0$ be given by Fact \ref{fact pA flow with all the features} item (\ref{factpAfeatures(iii)}). There exists $Q>0$ such that, for any $L \in \F$ and any pair $g_1, g_2$ of geodesics in $L$ that intersect at angle $\theta \in [\theta_0, \pi/2]$, the intersection $B_L(g_1, C) \cap B_L(g_2, C)$ of their $C$-neighborhoods in $L$ has diameter smaller than $Q$. Now given $x \in L$, $\ell^u$, and $\ell^s$, if $g_{\ell^u} \cap g_{\ell^s} = \{x'\}$, then $x \in B_L(g_{\ell^u}, C) \cap B_L(g_{\ell^s}, C)$. Hence, $d_L(x, x') < Q$.
\end{proof}

\begin{lemma}
\label{lemma distance comparison Gs and Gu}
    For every $K>0$, there exists $R>0$ such that for every leaf $L \in  \F$ and every $x \in L$, if $y \in \mathcal G^u(x)$ satisfies $d_{\mathcal G^u}(x, y) > R$, then $d_{L}(\mathcal G^s(x), \mathcal G^s(y)) > K$. Similarly, if $y \in \mathcal G^s(x)$ satisfies $d_{\mathcal G^s}(x, y) > R$, then $d_{L}(\mathcal G^u(x), \mathcal G^u(y)) > K$. 
\end{lemma}

\begin{proof}

    Fixing $K$, let $d'>0$ be given by Fact \ref{fact angle geodesics hyperbolic plane} for $d = K + 2C$ and $\theta = \theta_0$ given by Fact \ref{fact pA flow with all the features} item (\ref{factpAfeatures(iii)}). Let $Q>0$ be the constant from Lemma \ref{lemma distance between intersections}. Choose $R > 0$ such that in every leaf $L$, if $x, y \in L$ are on the same unstable line and $d_{\mathcal G^u}(x, y) > R$, then $d_L(x, y) > d' + 2Q$ (such $R>0$ exists by Fact \ref{fact pA flow with all the features} item (\ref{factpAfeatures(ii)})). 
    
    In that case, $d_L(\mathcal G^s(x), \mathcal G^s(y)) > K$. Indeed, if $\ell^u$ is an unstable line connecting $x$ and $y$, and $\ell_x$ and $\ell_y$ are stable lines through $x$ and $y$, respectively, let $x' = g_{\ell_x} \cap g_{\ell^u}$ and $y' = g_{\ell_y} \cap g_{\ell^u}$. Then \[ d_L(x', y') \geq d_L(x, y) - 2Q > d', \] so \[ d_L(\ell_x, \ell_y) \geq d_L(g_{\ell_x}, g_{\ell_y}) - 2C > K. \]
    
    Analogously if $y \in \mathcal G^s(x)$ satisfies $d_{\mathcal G^s}(x, y) > R$.
\end{proof}

The periodic orbits of $\phi_t$ will play an important role henceforth, so it is convenient to record some of their properties here. Suppose $\delta$ is a regular periodic orbit of $\phi_t$ and let $\widetilde \delta \subset \widetilde M$ be a connected component of its preimage under the covering map. Note that $\widetilde \delta$ is a regular orbit of $\widetilde \phi_t$ which intersects each leaf $L \in \F$ at a unique point, denoted $x_L$. Taking a neighborhood $B$ of $\widetilde \delta$ with no singular points, for every $y\in B$ and $\ell(y)$ an unstable line trough $y$ we can define an arc-length parametrization $\{\alpha_{\ell(y)}(s)\}_{s \in \R}$ of $\ell(y)$, in such a way that $\alpha_{\ell(y)}(s)$ varies continuously with $y$ and every small $s$ (that is, we are fixing coherent arc-length parametrizations of the unstable lines through $B$). For each $L \in \F$, let us denote $\alpha_L$ to the curve $\alpha_{\mathcal G^u(x_L)}$. The following fact follows from the uniform continuity of $\alpha_L(t)$ with respect to initial conditions.

\begin{fact}
\label{fact continuity initial conditions}
    For every $T > 0$ and $\varepsilon > 0$, there exists $d > 0$ such that for every $L \in  \F$, if $y \in L$ satisfies $d_L(x_L, y) < d$, then $d_L(\alpha_L(t), \alpha_{\ell(y)}(t)) < \varepsilon$ for all $|t| \leq T$ and every $\ell(y)$ unstable line through $y$.
\end{fact}

An analogous fact to \ref{fact continuity initial conditions} holds for stable lines as well.

\vspace{2mm}

For a deck transformation $\gamma \in \pi_1(M)$, we denote by $g_\gamma$ the geodesic in $\mathbb{H}^3$ invariant under $\gamma$. Consider $\pi: \tilde M \to M$ the covering map. If $\alpha$ is a closed curve in $M$, there exists a deck transformation $\gamma \in \pi_1(M)$ that leaves invariant a connected component of $\pi^{-1}(\alpha)$. In this case, we say that $\gamma$ is \textit{associated with the curve $\alpha$}. The following is \cite[Ch. 8, Thm. 30]{GH}.

\begin{prop}
\label{prop gamma and eta do not share fixed points at infinity}
    Let $\phi_t : M \to M$ be a regulating flow for a uniform $\R$-covered foliation of $M$. If $\eta, \gamma \in \pi_1(M)$ are represented by non-freely homotopic periodic orbits of $\phi_t$, then $\gamma$ and $\eta$ do not share any fixed points at the boundary of $\mathbb H^3$. Moreover, $\phi_t$ does not have distinct freely homotopic orbits. 
\end{prop}

The fact that regulating pseudo-Anosov flows do not have distinct freely homotopic periodic orbits follows from, for instance, \cite[Prop. 2.24]{BFM} which implies that if $\phi_t$ has two distinct freely homotopic orbits, then, there is a lozenge, which forces the existence of freely homotopic periodic orbits with oposite orientation (and this forbids being regulating to a foliation).

\subsection{Existence of $T_\gamma$} 

Here we show the existence of the sets $T_\gamma$ posited in Proposition \ref{prop exists T sub gamma}. 

\subsubsection{Construction of good neighborhoods}

Let us fix $\gamma \in \pi_1(M)$ associated with a regular periodic orbit $\delta$ of $\phi$. Let $\widetilde \delta$ be the connected component of $\pi^{-1}(\delta)$ invariant under $\gamma$. For each leaf $L \in \F$, we denote by $x_L$ the point of intersection of $\widetilde \delta$ with $L$. Recall that for each leaf $L \in  \F$ and each point $x \in L$, we denote by $\mathcal G^s(x)$ and $\mathcal G^u(x)$ the intersection of $\mathcal {W}^{ws}(x)$ with $L$ and $\mathcal {W}^{wu}(x)$ with $L$, respectively, and $\alpha_L(s)$ as an arc-length parametrization of $\mathcal G^u(x_L)$ that varies continuously with $L$.

To prove Proposition \ref{prop exists T sub gamma}, we will follow a strategy similar to that in \cite[\S 8]{BFFP}. For every $r \in \R$, let $u_L^r = \alpha_L(r)$. For $r \neq 0$, let $I_L^r$ be the connected component of $L \setminus \mathcal{G}^s(u^r_L)$ containing $x_L$. Consider the sets $U_1^r, U_2^r \subset \widetilde M$ defined by $U_1^r = \bigcup_{L \in  \F} (L \setminus I_L^r)$ and $U_2^r = \bigcup_{L \in  \F} (L \setminus I_L^{-r})$. Similarly, if $\beta_L$ parametrizes $\mathcal G^s(x_L)$, let $v^r_L = \beta_L(r)$ and define $J_L^r$ as the component of $L \setminus \mathcal G^u(v^r_L)$ containing $x_L$. Then define $V_1^r = \bigcup_{L \in  \F} (L \setminus J_L^r)$ and $V_2^r = \bigcup_{L \in  \F} (L \setminus J_L^{-r})$.

We denote by $U^r$ the union $U_1^r \cup U_2^r$ and by $V^r$ the union $V_1^r \cup V_2^r$ (see Figure \ref{def U y V}).

\begin{figure}[ht]
    \centering
    \def\svgwidth{0.5\columnwidth}
    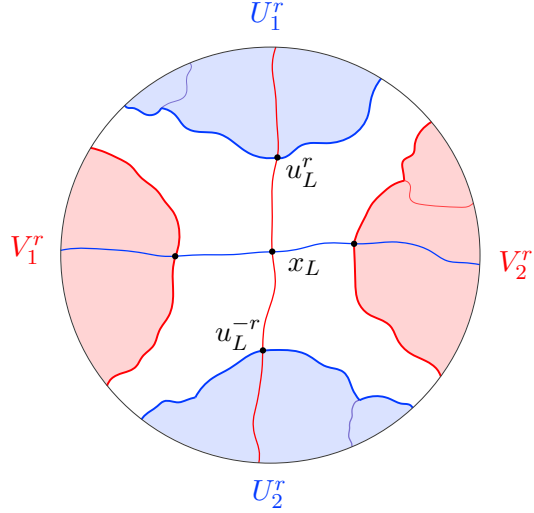
    \caption{$U^r$ and $V^r$ viewed on a leaf $L$.}
    \label{def U y V}        
\end{figure}

\begin{obs}
\label{obs buenos entornos}
For all $r > 0$, the following holds:
    \begin{enumerate}
        \item If $s > r$, then $U_i^s \subset U_i^r$ and $V_i^s \subset V_i^r$.
        \item $\gamma (V_i^r) = V_i^r$ and $\gamma (U_i^r) = U_i^r$.
        \item There exists $K > 0$ such that $\mathcal{G}^{u}(x_L) \setminus U^r$ and $\mathcal G^s(x_L) \setminus V^r$ are contained in $B_L(x_L, K)$, for all $L \in  \F$.
        \item $\mathcal W^{wu}(\widetilde \delta)$ does not intersect $V^r$ and $\mathcal W^{ws}(\widetilde \delta)$ does not intersect $U^r$.
    \end{enumerate}
\end{obs}

\begin{lemma}
\label{lemma continuity initial conditions with U}
For every $r > 0$, there exists $d > 0$ such that for all $L \in  \F$ and all $y \in \mathcal G^s(x_L)$ with $d_{\mathcal G^s}(x_L, y) < d$, we have $\mathcal G^u(y) \subset B_L(\mathcal G^u(x_L), 1) \cup U^r$.
\end{lemma}

\begin{proof}
    Note that in every leaf $L$, $\alpha_L(T) \in U_1^r$ for all $T > r$. Due to Lemma \ref{lemma distance comparison Gs and Gu}, for sufficiently large $T$, there exists $\varepsilon > 0$ such that $\varepsilon < d_L(\alpha_L(T), \mathcal G^s(u_L^r))$ for all $L \in \F$. By Fact \ref{fact continuity initial conditions}, there exists $d$ independent of $L$ such that for all $y \in L$ with $d_L(x_L, y) < d$, any unstable line $\ell(y)$ satisfies $\alpha_{\ell(y)}(T) \in U_1^r$. We can also assume that $d < r$, which implies $\alpha_{\ell(y)}(t_0) \in \mathcal G^s(u_L^r)$ for some $0 < t_0 < T$. Since $\mathcal G^u(y)$ intersects $\mathcal G^s(u_L^r)$ at most once (Fact \ref{fact no bighorns}), necessarily $\alpha_{\ell(y)}(t) \in U_1^r$ for all $t \geq T$.

    Similarly, in every leaf $L$, we have $\alpha_L(-T) \in U_2^r$. If we choose $\varepsilon$ such that $\varepsilon < d_L(\alpha_L(-T), \mathcal G^s(u_L^r))$, then $\alpha_{\ell(y)}(-t) \in U_2^r$ for all $t \geq T$. This implies that $\mathcal G^u(y) \in B_L(\mathcal G^u(x_L), \varepsilon) \cup U^r$. Requiring $\varepsilon < 1$, we conclude what we wanted.
    
\end{proof}

\subsubsection{Invariance of good neighborhoods}

Let $Z: \widetilde M / \F \to \widetilde M / \F$ be the structure map of $\cF$. Recall that given $x \in \widetilde M$, $\F(x)$ denotes the leaf of $ \F$ through $x$, and we write $\widetilde f^k > Z$ to denote that $ \F(\widetilde f^k(x))$ lies above $Z(\F(x))$ for all $x \in \widetilde M$. 

\begin{prop}
\label{prop existence of good neighborhoods for fixed k0}
Suppose $k$ is such that $\widetilde f^{k} > Z$. Then there exists $r_k \geq 1$ such that for $r \geq r_k$, $\widetilde f^{k}(U_i^r) \subset U_i^{r+1}$ and $\widetilde f^{-k}(V_i^r) \subset V_i^{r+1}$.
\end{prop}

We call $\phi_{f^k}: \widetilde M \to \widetilde M$ the map that sends a point $x$ to the intersection of the orbit of $\widetilde \phi_t$ through $x$ with the leaf through $\widetilde f^k(x)$. Observe that $\phi_{f^k}$ commutes with every deck transformation.

Let $\lambda > 1$ be a constant such that for every pair of leaves $L, L' \in \F$ with $L' \geq Z(L)$, the map $\tau_{L, L'}$ multiplies lengths (greater than 1) inside $\mathcal G^u$ by $\lambda$ and lengths (greater than 1) inside $\mathcal G^s$ by $\lambda^{-1}$ (cf. Proposition \ref{prop structure map contracts expands}).

\begin{lemma}
\label{lemma sequence of neighborhoods for the + flow exists}
Suppose $k$ satisfies $\widetilde f^k > Z$. Then for $i = 1, 2$ and every $r \geq 1$, $\phi_{f^k}(U_i^r) \subset U_i^{\lambda r}$ and $\phi_{f^{-k}}(V_i^r) \subset V_i^{\lambda r}$.
\end{lemma}

\begin{proof}
Given $x \in U_1^r$, let $L$ be the leaf through $x$ and $L'$ the leaf through $\widetilde f^k(x)$. Note that $\phi_{f^k}(x) = \tau_{L, L'}(x)$. From Proposition \ref{prop structure map contracts expands}, it follows that $d_{\mathcal G^u}(x_{L'}, \tau_{L, L'}(u_L^r)) > \lambda r$, so $\tau_{L, L'}(\mathcal G^s(u_L^r)) = \mathcal G^s(\tau_{L, L'}(u_L^r))$ is contained in $U_1^{\lambda r}$. This means $\tau_{L, L'}(U_1^r \cap L)$ is contained in $U_1^{\lambda r}$. In particular, $\phi_{f^k}(x) = \tau_{L, L'}(x) \in U_1^r$.

The proof for $U_2^r$ and $V_i^r$ is analogous.
    
\end{proof}

\begin{lemma}
\label{lema comparacion de f con el flujo}
    For every \( k \in \N \), there exists \( K_k \) such that \[ d_{ \F(\widetilde f^k(x))}(\widetilde f^k(x), \phi_{f^k}(x)) < K_k \] for all \( x \in \widetilde M \). 
\end{lemma}

\begin{proof}
    The function \( x \mapsto d_{ \F(\widetilde f^k(x))}(\widetilde f^k(x), \phi_{f^k}(x)) \) is continuous and \( \pi_1(M) \)-invariant. That is, for every \( x \in \widetilde M \) and \( \eta \in \pi_1(M) \), we have
    
        \begin{equation*}
        \begin{aligned}
             & d_{ \F(\widetilde f^k(\eta x))}(\widetilde f^k(\eta x), \phi_{f^k}(\eta x)) \\
             = & \ d_{ \F(\eta \widetilde f^k(x))}(\eta \widetilde f^k(x), \eta \phi_{f^k}(x)) \\
             = & \ d_{\eta  \F( \widetilde f^k(x))}(\eta \widetilde f^k(x), \eta \phi_{f^k}(x)) \\
             = & \ d_{ \F(\widetilde f^k(x))}(\widetilde f^k(x), \phi_{f^k}(x)),
        \end{aligned}
    \end{equation*}
    given that \( \widetilde f \) and \( \widetilde \phi \) commute with \( \eta \) and \( \eta \) is an isometry between \( L \) and \( \eta L \) for every \( L \in  \F \).

    Since \( M \) is compact, this implies that \( d_{\F(\widetilde f^k(x))}(\widetilde f^k(x), \phi_{f^k}(x)) \) is bounded on \( \widetilde M \).

\end{proof}


\begin{proof}[Proof of Proposition \ref{prop existence of good neighborhoods for fixed k0}]
    We assume \( k \) satisfies \( \widetilde f^k > Z \), so \( \phi_{f^k}(U_1^r) \subset U_1^{\lambda r} \) for all \( r \geq 1 \) (Proposition \ref{prop structure map contracts expands}).

    Let \( K_k \) be given by Lemma \ref{lema comparacion de f con el flujo}, and let \( R_k \geq 1 \) such that for all \( x, y \in \widetilde M \) with \( y \in \mathcal G^u(x) \) and \( d_{\mathcal G^u}(x, y) > R_k \), we have \( d(\mathcal G^s(x), \mathcal G^s(y)) > K_k \) (such \( R_k \) exists by Lemma \ref{lemma distance comparison Gs and Gu}).  

    Choose \( r_k \geq 1 \) such that \( \lambda r_k - (r_k + 1) > R_k \). Then, for all \( r \geq r_k \) and every \( L \in  \F \), we have \( d_{\mathcal G^u}(u_L^{\lambda r}, u_L^{r+1}) = \lambda r - (r+1) > R_k \), so \( d_L(L \cap U_1^{\lambda r}, L \setminus U_1^{r+1}) > K_k \) for every \( L \in \F \). 
    
    Suppose \( \widetilde f^k(U_1^r) \) is not contained in \( U_1^{r+1} \). Then there exists \( x \in U_1^r \) for which, if \( L' \) is the leaf through \( \widetilde f^k(x) \), we would have \[ d_{L'}(\widetilde f^k(x), \phi_{f^k}(x)) \geq d_{L'}(\widetilde f^k(x), U_1^{\lambda r}\cap L') > K_k, \] contradicting Lemma \ref{lema comparacion de f con el flujo}. Hence, we conclude that \( \widetilde f^k(U_1^r) \subset U_1^{r + 1} \)
    

    By similar reasoning, we prove that \( \widetilde f^k(U_2^r) \subset U_2^{r+1} \) and \( \widetilde f^{-k_0}(V_i^r) \subset V_i^{r+1} \) for all \( r \geq r_k \).

\end{proof}

\begin{remark}\label{rem-indepr_k}
We remark that while the construction of the neighborhoods $U_i^r$ and $V_i^r$ depends on the deck transformation $\gamma$ (equivalently, on the curve $\widetilde{\delta}$) the value of $k$ is independent on it, and just depends on the fact that the good lift $\widetilde f$ verifies $\widetilde f^k > Z$. Once this value of $k$ is fixed, then, we also get a fixed value of $r_k$ since its choice, made in the previous proof, depends only on $K_k$ and $\lambda$ and not on $\gamma$ (or $\widetilde{\delta}$). 
\end{remark}

%

\subsubsection{Escape properties}

Recall that $\gamma$ is a deck transformation associated with a regular periodic orbit $\delta$ of the pseudo-Anosov flow $\phi$, and $\widetilde{\delta}$ is the connected component of $\pi^{-1}(\delta)$ invariant under $\gamma$.

\begin{prop}
    \label{prop fuera de D me alejo de delta}
    Suppose $k$ is such that $\widetilde{f}^k \geq Z$, and $r_k$ is given by Proposition \ref{prop existence of good neighborhoods for fixed k0}. Then, for $i=1,2$, for every $x \in U_i^{r_k}$ and $y \in V_i^{r_k}$, we have $d(\widetilde{f}^{nk}(x), \widetilde{\delta}) \xrightarrow[n]{} +\infty$ and $d(\widetilde{f}^{-nk}(y), \widetilde{\delta}) \xrightarrow[n]{} +\infty$.
\end{prop}

The above proposition follows from the following two lemmas.

\begin{lemma}
    \label{lema dentro de U o V la distancia a xL tiende a infinito en L}
    Suppose $k$ satisfies $\widetilde{f}^k \geq Z$, and $r_k$ is given by Proposition \ref{prop existence of good neighborhoods for fixed k0}. Then, for every $x \in U^{r_k}$, if $L_n$ is the leaf through $\widetilde{f}^{nk}(x)$, we have $d_{L_n}(x_{L_n}, \widetilde{f}^{nk}(x)) \xrightarrow[n]{} +\infty$. Similarly, for every $y \in V^{r_k}$, if $L_n$ is the leaf through $\widetilde{f}^{-nk}(y)$, we have $d_{L_n}(x_{L_n}, \widetilde{f}^{-nk}(y)) \xrightarrow[n]{} +\infty$.
\end{lemma}

\begin{proof}
    Suppose $x \in U_1^{r_k}$. Lemma \ref{lemma distance comparison Gs and Gu} ensures that $d_{L_n}(x_{L_n}, U_1^{r_k + n} \cap L_n) \xrightarrow[n]{} +\infty$, since $d_{\mathcal{G}^u}(x_{L_n}, u_{L_n}^{r_k + n}) = r_k + n \xrightarrow[n]{} +\infty$. Proposition \ref{prop existence of good neighborhoods for fixed k0} guarantees that $\widetilde{f}^{nk}(x) \in U_1^{r_k + n}$, thus $d_{L_n}(x_{L_n}, \widetilde{f}^{nk}(x)) \xrightarrow[n]{} +\infty$.

    Similar arguments apply if $x \in U_2^{r_k}$ or $y \in V_i^{r_k}$.
\end{proof}

\begin{lemma}
    \label{lema comparacion bola en M con bola en hojas}
    For every $K > 0$, there exists $K' > 0$ such that $$B(\widetilde{\delta}, K) \subseteq \bigcup_{L \in \F} B_L(x_L, K').$$
\end{lemma}

\begin{proof}
    Since $\gamma$ is an isometry in $\widetilde{M}$ and $\widetilde{\delta}$ is $\gamma$-invariant, for a fixed $K > 0$, the neighborhood $B(\widetilde{\delta}, K)$ is $\gamma$-invariant. Moreover, the quotient of $\overline{B(\widetilde{\delta}, K)}$ by the action of the subgroup of $\pi_1(M)$ generated by $\gamma$ is compact (homeomorphic to a solid torus). Let $h: \widetilde{M} \to \mathbb{R}$ assign to any point $y$ on a leaf $L$ the distance $d_L(x_L, y)$. Since $h$ is continuous and $\gamma$-invariant, $h$ is bounded on $\overline{B(\widetilde{\delta}, K)}$, which proves the lemma.
\end{proof}

\begin{proof}[Proof of Proposition \ref{prop fuera de D me alejo de delta}]
    This follows directly from Lemmas \ref{lema dentro de U o V la distancia a xL tiende a infinito en L} and \ref{lema comparacion bola en M con bola en hojas}. 
\end{proof}

\subsubsection{Construction of invariant closed sets}

The objective of this part is to prove the following proposition.

\begin{prop}
\label{prop existe k1 que verifica todas}
    Up to replacing $\widetilde f$ by a high power, there exists $r_0$ such that,  for $i=1,2$, if $U_i = U_i^{r_0}$ and $V_i = V_i^{r_0}$, the following hold:

    \begin{enumerate}
        \item $\widetilde f(\overline{U_i}) \subset U_i$ and $\widetilde f^{-1}(\overline{V_i}) \subset V_i$,
        \item There exists $K > 0$ such that for every $L \in  \F$, the sets $L \setminus (U \cup \widetilde f(V))$ and $L \setminus (V \cup \widetilde f^{-1}(U))$ are contained in $B_L(x_L, K)$, and
        \item For every $x \in U_i$, $y \in V_i$, it holds that $d(\widetilde f^{n}(x), \widetilde \delta) \xrightarrow[n]{} +\infty$ and $d(\widetilde f^{-n}(y), \widetilde \delta) \xrightarrow[n]{} +\infty$.
    \end{enumerate}
\end{prop}

To achieve this, we first establish similar properties for the regulating flow.

\begin{lemma}
\label{lema existe t0 para el flujo}
    If $k$ satisfies $\widetilde f^k > Z$ and $r_k$ is given by Proposition \ref{prop existence of good neighborhoods for fixed k0}, let $U = U_1^{r_k} \cup U_2^{r_k}$ and $V = V_1^{r_k} \cup V_2^{r_k}$. Then there exist $N > 0$ and $j \geq 1$ such that both $\widetilde M \setminus (U \cup \phi_{f^{jk}}(V))$ and $\widetilde M \setminus (V \cup \phi_{f^{-jk}}(U))$ are contained in $\bigcup_{L \in \F} B_L(x_L,  N)$.
\end{lemma}


\begin{proof}
   Let $d > 0$ given by Lemma \ref{lemma continuity initial conditions with U} for $r = r_k$. Let $k' > 0$ given by Proposition \ref{prop structure map contracts expands} applied to $d$. Let $j' \geq 1$ such that $\lambda^{j'} d > r_k$. It follows that for every leaf $L$, if $L' \geq Z^{k'j'}(L)$, the map $\tau_{L,L'}$ takes segments of $\mathcal G^s(x_L)$ of length $r_k$ to segments of $\mathcal G^s(x_L')$ of length less than $d$. Set $j = k'j'$.

    Let $x \in \widetilde M \setminus V$, and let $L$ be the leaf through $x$ and $L'$ the leaf through $\widetilde f^{jk}(x)$. We will show that $\phi_{f^{jk}}(x) \in U \cup B_{L'}(\mathcal G^u(x_{L'}), 1)$. Notice that $\phi_{f^{jk}}(x) = \tau_{L, L'}(x)$. Since $\widetilde f^k > Z$, it follows that $L' > Z^j(L)$. Hence, $d_{\mathcal G^s}(\tau_{L, L'}(v_L^{r_k}), x_{L'}) < d$. Therefore, by Lemma \ref{lemma distance comparison Gs and Gu}, \[ \tau_{L, L'}(\mathcal G^u(v_L^{r_k})) = \mathcal G^u(\tau_{L, L'}(v_L^{r_k})) \subset B_{L'}(\mathcal G^u(x_{L'}), 1) \cup U. \] Thus, $\tau_{L, L'}(L \setminus V)$ is contained in $B_{L'}(\mathcal G^u(x_{L'}), 1) \cup U$, and therefore, $\phi_{f^{jk}}(x) \in B_{L'}(\mathcal G^u(x_{L'}), 1) \cup U$. Taking $N = r_k + 1$, we have $\widetilde M \setminus (\phi_{f^{jk}}(V) \cup U)$ contained in $\bigcup_{L \in \F} B_L(x_L, N)$. Similar reasoning applies to show that $\widetilde M \setminus (\phi_{f^{-jk}}(U) \cup V) \subset B_L(x_L, N)$.
    
\end{proof}

\begin{proof}[Proof of Proposition \ref{prop existe k1 que verifica todas}]
    Let $k$ be such that $\widetilde f^k > Z$ (Lemma \ref{lema existe k tq f a la k pasa Z}) and set $r_0 := r_k$ as given by Proposition \ref{prop existence of good neighborhoods for fixed k0}. Take $d > 0$ and $j \geq 1$ from Lemma \ref{lema existe t0 para el flujo}, and $K_{jk}$ from Lemma \ref{lema comparacion de f con el flujo}. Replace $\widetilde f$ by its iterate $\widetilde f^{jk}$. Defining $K = d + K_{jk}$, we have that $L \setminus (U^{r_0} \cup \widetilde f(V^{r_0}))$ and $L \setminus (V^{r_0} \cup \widetilde f^{-1}(U^{r_0}))$ are contained in $B_L(x_L, K)$ for every $L \in  \F$. Moreover, Proposition \ref{prop existence of good neighborhoods for fixed k0} guarantees that $\widetilde f(U_i^{r_0}) \subset U_i^{r_0+1}$ and $\widetilde f^{-1}(V_i^{r_0}) \subset V_i^{r_0+1}$, so $\widetilde f(\overline{U_i^{r_0}}) \subset U_i^{r_0}$ and $\widetilde f^{-1}(\overline{V_i^{r_0}}) \subset V_i^{r_0}$. Point 3 is satisfied by Proposition \ref{prop fuera de D me alejo de delta}.
    
\end{proof}

\begin{remark}\label{rem-indepr0}
Note that as in Remark \ref{rem-indepr_k} the choice of $r_0$ only depends on the choice of $k$ that was fixed before and is independent on $\gamma$. 
\end{remark}

\subsubsection{Conclusion of the proof}

\begin{proof}[Proof of Proposition \ref{prop exists T sub gamma}]
    First, observe that it suffices to prove the result for some iterate of $\widetilde f$. This is because, if $T'_\gamma \subset \widetilde M$ is a closed set invariant under both ${\widetilde f}^k$ and $\gamma$, and is the maximal invariant in the $r$-neighborhood of $g_\gamma$ for all $r \geq r_0'$, we define $T_\gamma = \bigcup_{j = -k}^k \widetilde f^j(T'_\gamma)$. Then $T_\gamma$ is a closed $\gamma$-invariant set, and $f(T_\gamma) = T_\gamma$. Define $r_0 = r'_0 + 2k d$, especially $T_\gamma$ is $\widetilde f^k$-invariant contained in the $r_0$-neighborhood of $g_\gamma$, so $T_\gamma = T'_\gamma$, thus $T'_\gamma$ is $\widetilde f$-invariant.
    
    Replace $\widetilde f$ by an iterate for which there exists $r_0$ such that points 1, 2, and 3 of Proposition \ref{prop existe k1 que verifica todas} hold. Let $K > 0$ such that for every leaf $L \in  \F$, both $L\setminus (\widetilde f(V) \cap U)$ and $L \setminus (\widetilde f^{-1}(U) \cap V)$ are contained in the ball $B_L(x_L, K)$, and $\mathcal G^u(x_L) \subset B_L(x_L, K) \cup U$ and $\mathcal G^s(x_L) \subset B_L(x_L, K) \cup V$ (to ensure this, it suffices that $K > r_0$).

For every leaf $ L \in \F $, for brevity we denote $ U_i^L = U_i \cap L $, $ V_i^L = V_i \cap L $, $ U^L = U \cap L $, and $ V^L = V \cap L $.

Consider
\[ D = \bigcup_{L \in \F} \overline{B_L(x_L, K)}. \]
and, for each $ n \in \N $,
\[ R^n = \bigcap_{k = 0}^n \widetilde f^k(D) \quad \textrm{and} \quad Q^n = \bigcap_{k = 0}^n \widetilde f^{-k}(D). \]
Notice that $ D $ is closed, and therefore, $ R^n $ and $ Q^n $ are also closed.

\begin{afir}
\label{afir Rn menos V es pared}
    If $ C \subset \widetilde M $ is a connected set intersecting $ V_1 $ and $ V_2 $ but not intersecting $ U $, then $ C $ intersects $ R^n \setminus V $ for every $ n \in \N $.

    Similarly, if $ C $ intersects $ U_1 $ and $ U_2 $ but not $ V $, then $ C $ intersects $ Q^n \setminus U $ for every $ n \in \N $.
\end{afir}

\begin{figure}[h]
    \centering
    \def\svgwidth{0.65\columnwidth}
\begingroup%
  \makeatletter%
  \providecommand\color[2][]{%
    \errmessage{(Inkscape) Color is used for the text in Inkscape, but the package 'color.sty' is not loaded}%
    \renewcommand\color[2][]{}%
  }%
  \providecommand\transparent[1]{%
    \errmessage{(Inkscape) Transparency is used (non-zero) for the text in Inkscape, but the package 'transparent.sty' is not loaded}%
    \renewcommand\transparent[1]{}%
  }%
  \providecommand\rotatebox[2]{#2}%
  \newcommand*\fsize{\dimexpr\f@size pt\relax}%
  \newcommand*\lineheight[1]{\fontsize{\fsize}{#1\fsize}\selectfont}%
  \ifx\svgwidth\undefined%
    \setlength{\unitlength}{271.6949176bp}%
    \ifx\svgscale\undefined%
      \relax%
    \else%
      \setlength{\unitlength}{\unitlength * \real{\svgscale}}%
    \fi%
  \else%
    \setlength{\unitlength}{\svgwidth}%
  \fi%
  \global\let\svgwidth\undefined%
  \global\let\svgscale\undefined%
  \makeatother%
  \begin{picture}(1,1.0068353)%
    \lineheight{1}%
    \setlength\tabcolsep{0pt}%
    \put(0,0){\includegraphics[width=\unitlength,page=1]{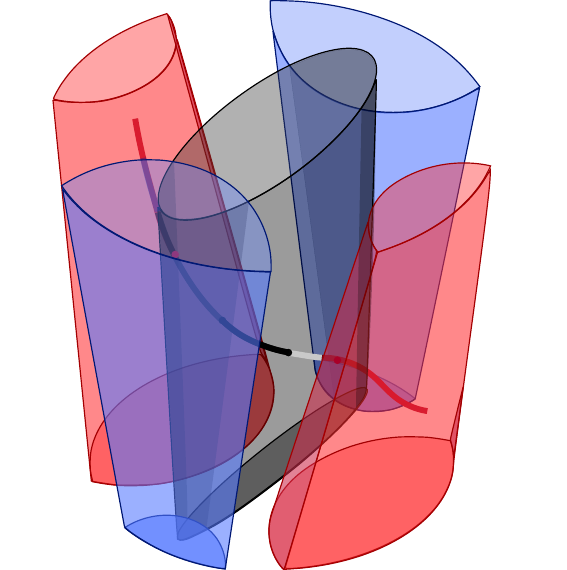}}%
    \put(0.09614323,0.09549236){\color[rgb]{0,0.09411765,0.45490196}\makebox(0,0)[lt]{\lineheight{10.25}\smash{\begin{tabular}[t]{l}$U_1$\end{tabular}}}}%
    \put(0.86992976,0.83022657){\color[rgb]{0,0.09411765,0.45490196}\makebox(0,0)[lt]{\lineheight{10.25}\smash{\begin{tabular}[t]{l}$U_2$\end{tabular}}}}%
    \put(0.90030034,0.49615081){\color[rgb]{0.66666667,0,0.15686275}\transparent{0.98039216}\makebox(0,0)[lt]{\lineheight{10.25}\smash{\begin{tabular}[t]{l}$V_2$\end{tabular}}}}%
    \put(-0.00274967,0.85358862){\color[rgb]{0.66666667,0,0.15686275}\transparent{0.98039216}\makebox(0,0)[lt]{\lineheight{10.25}\smash{\begin{tabular}[t]{l}$V_1$\end{tabular}}}}%
    \put(0.16673481,0.74612344){\color[rgb]{0.54117647,0,0.12941176}\transparent{0.98039216}\makebox(0,0)[lt]{\lineheight{10.25}\smash{\begin{tabular}[t]{l}$C$\end{tabular}}}}%
    \put(0.41096297,0.74326062){\color[rgb]{0,0,0}\makebox(0,0)[lt]{\lineheight{10.25}\smash{\begin{tabular}[t]{l}$R^n$\end{tabular}}}}%
  \end{picture}%
\endgroup%

    \caption{}
\end{figure}

\begin{proof}
    We will prove it by induction on $ n $, specifically for $ R^n $ as the proof for $ Q^n $ is analogous. Let $ C \subset \widetilde M $ be a connected set intersecting $ V_1 $ and $ V_2 $ but not $ U $.
    
    Note that $ R^0 = D $. Recall that $ \mathcal W^{wu}(\widetilde \delta) $ does not intersect $ V $ and is contained in $ D \cup U $. Since $ \mathcal W^{wu}(\widetilde \delta) $ separates $ \widetilde M $ into two components, one containing $ V_1 $ and the other $ V_2 $, we conclude that $ C $ must intersect $ \mathcal W^{cu} $. Therefore, $ C $ intersects $ D \setminus U $.
    
    Now suppose every connected set $ C' $ intersecting $ V_1 $ and $ V_2 $ but not $ U $ intersects $ R^{n-1} \setminus V $, and let $ C $ be such a set. Since $ \widetilde f^{-1}(V_i) \subset V_i $, it follows that $ \widetilde f^{-1}(C) $ is a connected set intersecting $ V_1 $ and $ V_2 $, and $ \widetilde f^{-1}(C) $ does not intersect $ U $. Thus, $ \widetilde f^{-1}(C) $ must intersect $ R^{n-1} \setminus V $, so $ C $ intersects $ \widetilde f(R^{n-1}) \setminus \widetilde f(V) $. From our choice of $ K $, the complement of $ \widetilde f(V) \cup U $ is contained in $ D $, hence $ C \cap \widetilde f(R^{n-1}) \setminus \widetilde f(V) \subseteq C \cap \widetilde f(R^{n-1}) \cap D $. Moreover, this intersection does not touch $ V $ since $ V $ is contained in $ \widetilde f(V) $. As $ R^n = \widetilde f(R^{n-1}) \cap D $, we conclude that $ C $ intersects $ R^n \setminus V $ as desired.
\end{proof}

Now, fixing $ n \in \N $ and $ L \in \F $, we denote $ R^n_L = R^n \cap L $ and $ Q^n_L = Q^n \cap L $. By Claim \ref{afir Rn menos V es pared}, we know that $ R^n_L \cup U^L \setminus V $ separates $ V^L_1 $ from $ V^L_2 $, whereas $ R^n \setminus V $ does not separate them. Therefore, there exists a connected component $ C $ of $ R^n_L \setminus V $ that intersects $ U_1 $ and $ U_2 $. Then, again by Claim \ref{afir Rn menos V es pared}, we have that $ C $ intersects $ Q^n_L $. In particular, we conclude that $ R_L^n $ intersects $ Q_L^n $.

Finally, we define
\[ T_\gamma = \bigcap_{n \geq 0} R^n \cap Q^n. \]
Then $ T_\gamma $ is $ f $-invariant by construction, and it is $\gamma$-invariant since $ D $ is and $\widetilde f$ commutes with $\gamma$. Moreover, the intersection of $ T_\gamma $ with every leaf $ L \in  \F $ is the non-empty compact set $ \bigcap_{n \geq 0} R^n_L \cap Q^n_L $.

On the other hand, observe that there exists $ r_0 $ such that $ T_\gamma $ is contained in the $ r_0 $-neighborhood of $ g_\gamma $, since $ \delta $ is homotopic to the projection of $ g_\gamma $. We will see that $ T_\gamma $ is the maximal invariant set in every $ r $-neighborhood of $ g_\gamma $ with $ r \geq r_0 $.

From point 2 of Proposition \ref{prop existe k1 que verifica todas}, it follows that $ \widetilde M \setminus D $ is contained in $ U \cup \widetilde f(V) $. Thus, point 3 of Proposition \ref{prop existe k1 que verifica todas} implies that every point $ x $ outside $ D $ satisfies $ d(\widetilde f^n(x), g_\gamma) \xrightarrow[n]{} +\infty $ or $ d(\widetilde f^{-n}(x), g_\gamma) \xrightarrow[n]{} +\infty $.

Now, suppose $ S \subset \widetilde M $ is $\widetilde f$-invariant and contained in the $ r $-neighborhood of $ g_\gamma $ for some $ r \geq r_0 $. From the previous observation, necessarily $ S $ is contained in $ D $, so $ S = \widetilde f^{n} (S) $ is contained in $ \widetilde f^{n}(D) $ for every $ n \in \Z $. Therefore,
\[ S \subset \bigcap_{n \geq 0} \widetilde f^n(D) \cap \widetilde f^{-n}(D) = T_\gamma. \]

Note that the choice of $r_0$ is independent of the deck transformation $\gamma$ and only depends on the properties of $\ft$ with respect to the regulating flow (see Remark \ref{rem-indepr0}).

This concludes the proof of Proposition \ref{prop exists T sub gamma} up to Proposition \ref{prop-Tgammacon} below. 

\end{proof}

\subsection{Further properties of $T_\gamma$}\label{ss.furtherprops}
Here we show the following: 

\begin{prop}\label{prop-Tgammacon}
The set $T_\gamma$ contains a connected set which intersects every leaf of $\F$. 
\end{prop}

\begin{proof}
Recall that $T_\gamma =\bigcap_n (R^n \cap Q^n)$ where: 

\[ R^n = \bigcap_{k = 0}^n \widetilde f^k(D) \quad \textrm{and} \quad Q^n = \bigcap_{k = 0}^n \widetilde f^{-k}(D). \]

We fix $\eps>0$, then, there is $n_0$ so that for every $n>n_0$ we have that $R^n \cap Q^n$ is contained in an $\eps$-neighborhood of $T_\gamma$, so, to prove the result, it is enough to show that for every $n$ there is a properly embeded line in $R^n \cap Q^n$ which, since $R^n \cap Q^n$ intersects each leaf $L \in \F$ in a bounded set, must intersect every leaf of $\F$. 

Recall that $\mathcal W^{wu}(\widetilde \delta)$ is contained in $D\cup U$ and avoids $V$ and $\mathcal W^{ws}(\widetilde \delta)$ is contained in $D\cup V$ and avoids $U$. It follows that $R^n \cup V$ and $Q^n \cup U$ contain properly embedded surfaces avoiding respectively $U$ and $V$. Up to small perturbation, we can assume that these surfaces, which we denote as $A^{u}$ and $A^s$ respectively, intersect transversally and their intersection is then contained in $R^n \cap Q^n$. By the transversality $A^u \cap A^s$ must be a union of properly embedded curves in $R^n \cap Q^n$ which separates $A^u \cap V_1$ from $A^u \cap V_2$ we deduce that at least one connected component must intersect every leaf of $\F$. This is the desired component that allows us to conclude the proof.   
\end{proof} 

In particular, the previous proposition allows one to show that, when projected to $M$, the sets $T_\gamma$ separate the homology of $M$ in a similar way as a circle would, in the sense that if $L \in \F$ and $\eta$ is a closed curve in the proyection of $L$ to $M$ that, when lifted to $\mt$ surounds $T_\gamma \cap L$ we get that $\eta$ is homotopically non-trivial in $M$. There are other ways to understand the topology of the complement of the projection of $T_\gamma$ in $M$ that we will not pursue. A natural question that arises from the argument above is whether $T_\gamma$ contains curves intersecting every leaf of $\F$. This may not be true.

We close this section by showing that the sets $T_\gamma$ can be quite wild. 

\begin{example}\label{ex:pseudocircle}
Start with a quasi-geodesic flow $\varphi_t : M \to M$ of a closed hyperbolic 3-manifold. For instance, one can take the suspension flow of a pseudo-Anosov homeomorphism of a surface of genus $\geq 2$. Fix a regular periodic orbit, in a small neighborhood it looks like the suspension of a hyperbolic fixed point in the plane. By considering a convenient time $t_0$, if we consider $f_0 = \varphi_{t_0}$ one can take coordinates $(x,y,t) \in [-1,1]^2 \times S^1$ around the periodic orbit so that the dynamics of $f_0$ in these coordinates is: $(x,y,t)=(2x, \frac{1}{2} y, t)$ (here, $t \in S^1$ where $S^1$ is considered as $\RR/_{\ZZ}$ and we think of this as making an integer amount of turns). We will remove a small neighborhood of $x=y=0$ and change the dynamics so that the maximal invariant set in this neighborhood is a pseudo-circle. Doing this, one gets $T_\gamma$ to project to a pseudo-circle if $\gamma$ is the deck transformation associated to this periodic orbit. For this, consider $h: [-1,1] \times S^1 \to [-1,1] \times S^1$ an embedding with the following properties: 
\begin{itemize}
\item $h(y,t) = (\frac{1}{2}y , t)$ in a neighborhood of the boundary of $[-1,1]\times S^1$. 
\item $\bigcap_{n>0} h^{n}([-1,1] \times S^1)$ is a pseudo-circle where $h$ acts as the identity. 
\end{itemize}
Such a construction can be found for instance in \cite{CO}. Now, since $h$ is homotopic to the identity, we can choose $h_s:[-1,1] \times S^1 \to [-1,1] \times S^1$ with $s\in [-1,1]$ a continuous family of embeddings so that $h_s(y,t)= (\frac{1}{2}y , t)$ in a neighborhood of the boundary of $[-1,1]\times S^1$ for all $s$, and so that $h_0=h$ and $h_{\pm 1}(y,t)=(\frac{1}{2}y , t)$ for all $(y,t)$. Finally, we can modify $f_0$ to $f_1$ which coincides with $f_0$ outside this neighborhood and so that in these coordinates it is $f_1(x,y,t)=(2x, h_x(y,t))$. This has the desired properties. 
\end{example}

\subsection{Disjointness of the invariant sets}

Let us begin by noting that the invariant closed sets $T_\gamma$ project to closed sets in $M$.

\begin{lemma}
\label{lemma Tgamma projects to compact}
    For every $\gamma \in \pi_1(M)$ associated with a regular periodic orbit of $\phi$, the projection $\pi(T_\gamma)$ is compact.
\end{lemma}

\begin{proof}
    It suffices to show that $\pi^{-1}(\pi(T_\gamma))$ is closed in $\widetilde M$, as then $\pi(T_\gamma)$ will be closed in $M$, which is compact. We have \[ \pi^{-1}(\pi(T_\gamma)) = \bigcup_{\eta \in \pi_1(M)} \eta T_\gamma. \] Since the sets $\eta T_\gamma$ are closed in $\widetilde M$, it suffices to prove that any compact set $K \subset \widetilde M$ intersects only finitely many of them. Proposition \ref{prop exists T sub gamma} ensures that $T_\gamma$ lies within the $r$-neighborhood of the geodesic $g_\gamma$ for some $r > 0$. Given $K$, let $B$ be a ball in $\widetilde M$ containing $K$, large enough so that the distance from $K$ to the boundary of $B$ is greater than $r$. Then, if the geodesic $\eta g_\gamma$ is disjoint from $B$, the translation $\eta T_\gamma$ does not intersect $K$. Since the action of $\pi_1(M)$ on $\widetilde M$ is discrete and $\pi(g_\gamma)$ is closed in $M$, only finitely many translations $\eta g_\gamma$ intersect $B$, so only finitely many translations of $T_\gamma$ intersect $K$.
    
\end{proof}

Proposition \ref{prop if eta and gamma different, T projections are disjoint} follows from the next two lemmas.

\begin{lemma}
\label{lema si eta y gama distintas, los T son disjuntos}
    Let $\gamma, \eta \in \pi_1(M)$ be associated with regular periodic orbits of $\phi$, such that $\langle \eta \rangle \cap \langle \gamma \rangle = \{1\}$. Then $T_\gamma$ and $T_\eta$ are disjoint.
\end{lemma}

Note that if $\gamma \in \pi_1(M)$ is associated with a periodic orbit of $\phi$, then for every $\eta \in \pi_1(M)$, the element $\eta \gamma \eta^{-1}$ is associated with the same periodic orbit of $\phi$. Therefore, by Proposition \ref{prop exists T sub gamma}, it makes sense to write $T_{\eta \gamma \eta^{-1}}$.

\begin{lemma}
\label{lema trasladar por eta un T me da otro T}
    If $\gamma \in \pi_1(M)$ is associated with a regular periodic orbit of $\phi_t$, then for every $\eta \in \pi_1(M)$ one has $\eta T_\gamma = T_{\eta \gamma \eta^{-1}}$.
\end{lemma}

\begin{proof}[Proof of Proposition \ref{prop if eta and gamma different, T projections are disjoint}]

    Let $\gamma, \eta$ be associated with regular periodic orbits $o_\gamma$ and $o_\eta$ respectively.
    
    We have \[ \pi^{-1}(\pi(T_\gamma)) = \bigcup_{\alpha \in \pi_1(M)} \alpha T_\gamma \quad \textrm{ and } \quad \pi^{-1}(\pi(T_\eta)) = \bigcup_{\alpha \in \pi_1(M)} \alpha T_\eta \]
    
    If $\pi(T_\gamma) \cap \pi(T_\eta) \neq \emptyset$, then there exists $\alpha \in \pi_1(M)$ such that $\alpha T_\gamma$ intersects $T_\eta$. By Lemma \ref{lema trasladar por eta un T me da otro T}, this means $T_{\alpha \gamma \alpha^{-1}}$ intersects $T_\eta$. From Lemma \ref{lema si eta y gama distintas, los T son disjuntos}, we deduce that $\langle \alpha \gamma \alpha^{-1} \rangle \cap \langle \eta \rangle \neq \{1\}$. Since $\eta$ and $\gamma$ are primitive, it follows that $\eta$ must be conjugate to $\gamma$ or its inverse. This implies $o_\gamma$ is homotopic to $o_\eta$ or its inverse, which, according to Proposition \ref{prop gamma and eta do not share fixed points at infinity}, implies $o_\gamma = o_\eta$. 

\end{proof}

\begin{proof}[Proof of Lemma \ref{lema si eta y gama distintas, los T son disjuntos}]
    If $\langle \gamma \rangle \cap \langle \eta \rangle = \{1\}$, according to Proposition \ref{prop gamma and eta do not share fixed points at infinity}, the geodesics $g_\gamma$ and $g_\eta$ do not share any limit points at the boundary of $\mathbb{H}^3$. In addition, there exists $r$ such that $T_\gamma$ is contained in the $r$-neighborhood of $g_\gamma$ and $T_\eta$ in the $r$-neighborhood of $g_\eta$, as guaranteed by Proposition \ref{prop exists T sub gamma}. Suppose by contradiction that there exists $x \in T_\gamma \cap T_\eta$. Since $f$ has positive escape rate with respect to $\cF$, the orbit of $x$ under $\ft$ escapes every compact set in $\mt$ (Corollary \ref{coro escape de compactos}). As the boundary points of $g_\gamma$ and $g_\eta$ are distinct, the intersection of their $r$-neighborhoods is bounded, but $T_\gamma \cap T_\eta$, which is contained in this intersection, is also $\tilde f$-invariant, leading to a contradiction.
    
\end{proof}

\begin{proof}[Proof of Lemma \ref{lema trasladar por eta un T me da otro T}]
    
    To see that $\eta T_\gamma = T_{\eta \gamma \eta^{-1}}$, it suffices to show that $\eta T_\gamma$ is contained in $T_{\eta \gamma \eta^{-1}}$ for every $\eta \in \pi_1(M)$ and $\gamma \in \pi_1(M)$ associated with a periodic orbit of $\phi$. Indeed, in this case, we have that $\eta^{-1} T_{\eta \gamma \eta^{-1}} \subset T_\gamma$, and thus $T_{\eta \gamma \eta^{-1}} \subset \eta T_\gamma$.

    First of all, $\eta T_\gamma$ is invariant under $\eta \gamma \eta^{-1}$, since \[ (\eta \gamma \eta^{-1}) \eta T_\gamma = \eta \gamma T_\gamma = \eta T_\gamma, \] given that $T_\gamma$ is $\gamma$-invariant. Since $T_\gamma$ is $\tilde f$-invariant and $\tilde f$ is a good lift, it also holds that \[ \tilde f \eta T_\gamma = \eta \tilde f T_\gamma = \eta T_\gamma, \] so $\eta T_\gamma$ is $\tilde f$-invariant.

    To prove that $\eta T_\gamma \subset T_{\eta \gamma \eta^{-1}}$, it suffices to show that there exists $R_0$ such that, for all $r > R_0$, the maximal closed set invariant under $\eta \gamma \eta^{-1}$ and under $\widetilde f$ in the $r$-neighborhood of $g_{\eta \gamma \eta^{-1}}$ is $\eta T_\gamma$. On one hand, there exists $R_0$ such that, for all $r > R_0$, the maximal closed set invariant under $\eta$ and under $\tilde f$ within the $r$-neighborhood of $g_\gamma$ is $T_\gamma$. Note that $g_{\eta \gamma \eta^{-1}} = \eta g_\gamma$, therefore, for $r > R_0$, we know that $\eta T_\gamma$ is an invariant closed set in the $r$-neighborhood of $g_{\eta \gamma \eta^{-1}}$. Furthermore, for any closed set $C$ within this neighborhood, $\eta^{-1}C$ is a closed set contained in the $r$-neighborhood of $g_\gamma$. If $C$ is invariant under $\eta \gamma \eta^{-1}$ and under $\tilde f$, then \[ \gamma \eta^{-1} C = (\eta^{-1} \eta) \gamma \eta^{-1} C = \eta^{-1}C,  \] meaning that $\eta^{-1} C$ is $\gamma$-invariant. Also, $\tilde f \eta^{-1} C = \eta^{-1} \tilde f C = \eta^{-1} C$, so $\eta^{-1}C$ is $\tilde f$-invariant. Therefore, we have $\eta^{-1}C \subset T_\gamma$, which implies $C \subset \eta T_\gamma$. From the above, it follows that $\eta T_\gamma$ is the maximal closed invariant set contained in the $r$-neighborhood of $\gamma_{\eta \gamma \eta^{-1}}$, so $\eta T_\gamma \subset T_{\eta \gamma \eta^{-1}}$.

\end{proof}

\subsection{Extending to the closure and uncountably many sets}\label{ss-exclosure}

In this section we extend the construction to obtain, for each orbit $o$ of the regulating pseudo-Anosov flow $\widetilde{\phi_t}: \mt \to \mt$ an $\ft$-invariant closed set $T_o$ in $\mt$ which remains at uniformly bounded distance from $o$ and intersects every leaf of $\F$. 

Since $\phi_t : M \to M$ contains uncountably many disjoint compact invariant sets\footnote{It is well known that a pseudo-Anosov flow contains a suspension of a full shift (see for instance \cite{Ia}), and these in turn contain uncountably many disjoint minimal sets \cite{MH}.} , this shows that $f$ will also have uncountably many such sets. 

So, to complete the proof of Theorem \ref{teoC} (and of Theorem \ref{teo.B1}) we need to show: 

\begin{prop}
For every orbit $o$ of $\widetilde \phi_t$ there is a closed $\ft$-invariant set $T_o$ contained in the $r_0$-neighborhood of $o$ and containing a connected set which intersects every leaf of $\F$. 
\end{prop}

\begin{proof}
Consider a sequence of periodic orbits approximating the projection of the orbit $o$ to $M$ as in Proposition \ref{prop periodic orbits}. It follows that these correspond to orbits $o_n \to o$ uniformly on compact sets and such that $o_n$ are invariant under some deck transformations $\gamma_n \in \pi_1(M)$. For each $\gamma_n$ we obtain a closed set $T_{\gamma_n}$ contained in the $r_0$-neighborhood of $o_n$. The Hausdorff limit of the sets $T_{\gamma_n}$ in compact subsets produces the set $T_o$ with the desired properties (recall that $T_{\gamma_n}$ contains connected sets intersecting every leaf of $\F$ and this passes to the closure too). 
\end{proof}

\begin{remark}
An alternative way to see this result is to look at the maps $e^+$ and $e^-$ defined in \S \ref{s.QG}. We get that the set $\Lambda_f$ defined in Lemma \ref{lem.defLambda} must contain the set $\Lambda_\phi$ defined similarly for the quasi-geodesic homeomorphism obtained as the time one map of $\phi_t$. This follows by the continuity of these maps and the fact that periodic orbits are dense in $\Lambda_\phi$ and belong to $\Lambda_f$ because are given by points in the corresponding $T_\gamma$. Note that in general $\Lambda_f$ can be strictly larger than $\Lambda_\phi$. 
\end{remark}

\section{Homological rotation and finite lifts}\label{s.homological}

\subsection{Uniform positive rate and invariant measures}\label{ss.rateQM}
Here we prove Proposition \ref{prop-QM}. We consider then a homeomorphism $f: M \to M$ homotopic to the identity of a hyperbolic 3-manifold and a function $Q: \mt \to \RR$ satisfying \eqref{eq:QM}. We assume that for every $x \in \mt$ one has that $Q(\ft^n(x)) \to +\infty$. 

We first show that (compare with Lemma \ref{lema existe k tq f a la k pasa Z}): 

\begin{lema}\label{lemQM1}
For every $k>0$ there is $n_0$ so that if $n>n_0$ and $x\in \mt$ we have that $Q(\ft^n(x)) - Q(x)>k$. 
\end{lema}

\begin{proof}
Note that it is enough to show that because of \eqref{eq:QM} there is $n_1$ so that for every $x \in \mt$ there is some $1\leq i \leq n_1$ such that $Q(\ft^i(x)) - Q(x)> k$. If this holds, note that we can find some $m$ so that $|Q(\ft^j (x))- Q(x)| < mk$ for every $x \in \mt$ and $1\leq j\leq n_1$. This way, if we consider $n_0 > (m+1)n_1$ we obtain the desired statement.

To prove the existence of $n_1$ as above, we proceed by contradiction. If no such $n_1$ exists, then, there is a sequence $x_n \in \mt$ of points such that $Q(\ft^j(x_n)) - Q(x_n) < k$ for every $1\leq j \leq n$. 

By compactness of $M$ there are deck transformations $\gamma_n \in \pi_1(M)$ so that $\gamma_n x_n \to x_\infty \in \mt$. Using property \eqref{eq:QM} we know that $Q(\ft^j( \gamma x_n)) - Q(\gamma x_n) < k + K$ for every $1 \leq j \leq n$ and therefore, taking limits we deduce that $Q(\ft^n(x_\infty)) \leq k + K$ for every $n\geq 0$, a contradiction. 
\end{proof}

Note that the above directly implies that $\liminf_n \frac{1}{n} Q(\ft^n(x))\geq \lambda$ (where $\lambda = \frac{k}{n_0}$) so, to complete the proof of Proposition \ref{prop-QM}, it is enough to prove the quasi-geodesic property for $f$. 

\begin{proof}[Proof of Proposition \ref{prop-QM}]
As mentioned, Lemma \ref{lemQM1} proves the first part of the statement, it is thus enough to get the quasi-geodesic property. Since $\ft$ is a good lift, the lower bound $d(\ft^n(x),x) \leq \beta n + \beta$ in \eqref{eq:QG} is automatic for some large $\beta$, so we just need to obtain the lower bound. 

For this, we must show that if $\beta>0$ is large enough we have that $ d(\ft^n(x),x) \geq \beta^{-1} n - \beta$ for every $n>0$, and for this is enough, thanks to Lemma \ref{lemQM1}, to show that there exists $T>0$ so that:  

$$ Q(x)-Q(y) > n T  \Rightarrow d(x,y) > n . $$

Note that by induction, it is enough to show that there is $T>0$ so that $Q(x)-Q(y)>T$ implies that $d(x,y)>1$. Indeed, if such a $T>0$ exists and $Q(x)-Q(y)>nT$, then one can consider $z$ in the geodesic joining $x$ and $y$ so that $Q(x)-Q(z)>T$ and $Q(z)-Q(y)>(n-1)T$, and then obtain that $d(x,y)>1+d(z,y)$. Inductively one gets that $d(x,y)>n$. 

The existence of $T>0$ so that $Q(x)-Q(y)>T$ implies that $d(x,y)>1$ follows from \eqref{eq:QM}, indeed, if there are sequences $x_n, y_n$ for which $Q(x_n)-Q(y_n) \to \infty$ but $d(x_n,y_n)\leq 1$ then by compactness of $M$ we can consider deck transformations $\gamma_n$ which, up to taking subsequences, satisfy that $\gamma_n x_n \to x_\infty$ and $\gamma_n y_n \to y_\infty$. On the other hand, since $Q(x_n)- Q(y_n) \to \infty$, equation \eqref{eq:QM} implies that $Q(\gamma_n x_n) - Q(\gamma_n y_n) \to \infty$ contradicting the continuity of $Q$. This completes the proof. 
\end{proof}

\subsection{Homological rotation and fibers} 

Here we will assume some standard facts on hyperbolic 3-manifolds and the identification of the first cohomology and the second homology. We refer the reader to \cite[Chap. 5]{Calegari-book} for a general introduction. 

Our goal is to prove Theorem \ref{teo-homology}. We let then $f: M \to M$ be a homeomorphism homotopic to the identity of a hyperbolic 3-manifold $M$ and we let $c \in H^1(M, \RR)$ for which every invariant measure verifies that $\ell_c(\mu)>0$. Since invariant probablities form a compact set, there is no loss in generality in assuming that there is $c \in H^1(M,\ZZ)$ satisfying $\ell_c(\mu)>0$.


We can consider $[S] \in H_2(M, \ZZ)$ an integer homology class dual to $c$, by this, we mean that if $\alpha$ is a closed $1$-form with $[\alpha]=c$ then, it integrates positive on every closed curve which has positive intersection number with some representative of $S$. 

We consider $S \in [S]$ an  incompressible representative (this can be chosen by taking a minimal genus representative of $[S]$ as explained in \cite[Lemma 5.7]{Calegari-book}).  It is also possible to assume that the class is primitive, so that $S$ is connected.  

We will show the following result, which reduces Theorem \ref{teo-homology} to Theorem \ref{teoC}. 

\begin{prop}\label{prop.reduction}
Under the above assumptions, it follows that $S$ is a fiber, that is, $M \setminus S$ is homeomorphic to $S \times (0,1)$. In particular, there is a foliation of $M$ by leaves homeomorphic to $S$ which is uniform $\RR$-covered and for which $f$ has positive escape rate. 
\end{prop}

\begin{proof}
Since $S$ is non-trivial in $H_2(M,\ZZ)$ it cannot separate $M$ and thus after cutting $M$ along $S$ we get a connected manifold $N$ with two boundary components $S_1,S_2$, both homeomorphic to $S$. We want to show that $N$ is homeomorphic to $S \times [0,1]$. 

Note that when lifted to the universal cover $\mt$ of $M$ we get that $S$ lifts to a union of properly embedded planes $\{P_i\}_{i \in I}$ each of which separates $\mt$ into two open connected components $P_i^+$ and $P_i^-$ according to the orientation. 


Recall that, given $x \in M$, we denoted $\eta_x^n$ the arc from $x$ to $f^n(x)$ produced by the homotopy from $\mathrm{id}$ to $f$ (cf. \S~\ref{ss.homological}). The fact that for every measure we have that $\ell_c(\mu)>0$ implies that there exists $n>0$ so that the signed intersection number between $\eta_x^n$ (with fixed endpoints) and $S$ is $\geq 2$ by an argument as in Proposition \ref{prop-QM}.  If $\tilde x$ is a lift of $x$, we denote by $\tilde \eta_x^n$ to the lift of $\eta_x^n$ starting at $\tilde x$ (and ending at $\ft^n(\tilde x)$). 


We claim that this implies that we can order the lifts $P_i$ as follows. For every $i \in I$ and $\tilde x \in P_i^+$ (lifting $x \in S$) there is $i'$ such that $\ft^n(\tilde x) \in P_{i'}^+$ and $P_{i'}^+\subsetneq P_i^+$. We note that $i'$ may not be unique, but there is a $j=j(x)$ so that $P_{i'}^+\subset P_j^+$ for every such an $i'$. We note that $j(x)$ a priori depends on $\tilde x$, but the set of points $y \in P_i$ so that $j(y)=j$ is open, so, by connectedness, there is a unique $j$ so that $j=j(x)$ for every $\tilde x \in P_i$. We call this $j$ the \emph{succesor} of $i$. Arguing for $\ft^{-n}$ we also get that every $i$ has a predecesor. 

It follows that there is an ordering of the sets $P_i$ and so we can write $I= \ZZ$ and have $\{P_i\}_{i \in \ZZ}$ with $P_{i+1}$ being the succesor of $P_i$. Note that the action of $\pi_1(M)$ verifies that the stabilizer of $P_i$ is isomorphic to $\pi_1(S)$ and the action on the set of planes $P_i$ is an action by translations whose kernel is exactly $\pi_1(S)$. This implies that $\pi_1(S) \lhd \pi_1(M)$ and this implies that $M$ can be written as a suspension of $S$, that is, $N$ is homeomorphic to $S \times [0,1]$  as we wanted to show.


\end{proof}

\section{Non hyperbolic manifolds and some problems}\label{s.generalmanifolds}

When the manifold is not hyperbolic, the notion of quasi-geodesic homeomorphisms still makes sense, but it is certainly less powerful as the Morse lemma is no longer valid. There are known quasi-geodesic flows (thus, quasi-geodesic homeomorphisms by considering their time one maps) in Seifert manifolds and some with non-trivial JSJ decomposition (see \cite{ChandaFenley}) but these have not been fully explored, not even in the flow case. 

Known minimal homeomorphisms and flows in 3-manifolds are very few. The known examples include linear flows on tori and nilmanifolds, horocyclic flows\footnote{Note that the strong stable/unstable foliation of an Anosov flow is minimal and it is not a suspension flow. Using \cite{Shwartz} it is possible to show that if one parametrizes this foliation by a flow and takes any non-zero time of this flow, it will induce a minimal homeomorphism.} induced by Anosov flows (see \cite{HT} for a discussion) or minimal homeomorphisms in Seifert manifolds constructed by the Anosov-Katok method (\cite{FH}). Recently, some new examples appeared with the construction of new partially hyperbolic diffeomorphisms, whose strong stable or unstable foliations can provide examples of minimal flows and homeomorphisms (by taking a suitable time $t$ of the flow), see \cite{BGHP,BFP}. Note that the examples in \cite{FH} can be made to have positive escape rate, while the ones which are normalized by some dynamics which contracts the orbits cannot have such a rate. We mention also the examples constructed in \cite{BCLR} with the \emph{Denjoy-Rees technique} which starts from minimal or uniquely ergodic examples and make deformations that can have very wild behavior keeping the minimality or unique ergodicity (for instance, they can have positive entropy).

With the same ideas as in \S \ref{s.positivefol} one can show: 

\begin{teo}\label{teo.nonhyp}
Let $M$ be a closed 3-manifold admitting a uniform $\RR$-covered foliation $\cF$ and having some atoroidal piece in its JSJ decomposition. If $f: M \to M$ is a homeomorphism homotopic to the identity and $\ft$ is a good lift verifying that every orbit escapes at positive speed with respect to $\F$, then, we have that $f$ has uncountably many compact disjoint $f$-invariant sets and positive topological entropy. 
\end{teo}

We will not give the details of the proof of this theorem, we just notice that in \cite{FenleyRcoveredtransverse} the analogous regulating flow is produced and that in \cite{FP2} the techniques from \cite{BFFP} that we have extended here are pushed to this more general case.

We close the paper by posing some questions pointing towards understanding minimal homeomorphisms and flows in 3-manifolds which we find concrete enough to try to address. One simplifying assumption could be to assume that the dynamics in question preserves some two dimensional foliation. Even with that assumption, we do not know how to answer the following question. 

\begin{question}
Can one classify minimal homeomorphisms of 3-manifolds which preserve a two dimensional minimal foliation? Do they verify $\lim_n \frac{1}{n} d_{\mt}(\ft^n(x),x) =0$ for every $x \in \mt$?
\end{question}

In fact, the following is already unknown to us: 

\begin{question}
Let $\phi_t : M \to M$ be a minimal flow on a 3-manifold with fundamental group of exponential growth and which is tangent to a two dimensional foliation. Is it true that in the universal cover $\mt$ for every $x \in \mt$ one has that  $\lim_t \frac{1}{t} d_{\mt}(\widetilde{\phi_t}(x),x) =0$? 
\end{question}

\end{document}